\documentclass[a4paper]{amsart}
\usepackage{latexsym}
\usepackage{eepic,epic, amsmath}
\usepackage{amssymb}
\usepackage{tikz}
\usepackage{graphicx}

\newcommand{\NN}{{\mathbb N}} 
\newcommand{\sym}{\mathcal{S}} 
\newcommand{\diag}[2]{\left(\begin{smallmatrix}0&{#2}\\ {#1}&0\end{smallmatrix}\right)}
\newcommand{\di}{\text{\ensuremath{\diamond}}}
\newcommand{\eL}{\text{L}} 
\newcommand{\cF}{\mathcal{F}}
\newcommand{\cFk}{\mathcal{F}^{(k)}}
\newcommand{\cM}{\mathcal{M}}
\newcommand{\cMk}{\mathcal{M}^{(k)}}
\newcommand{\mcab}{M_{312}} 
\newcommand{\mbca}{M_{231}} 
\newcommand{\rev}{\overline} 
\newcommand{\cC}{\mathcal{C}} 
\newcommand{\new}{{\text{new}}}
\newcommand{\id}{\text{id}}   

\DeclareMathOperator{\st}{\mathrm{st}}

\theoremstyle{plain}
\newtheorem{theorem}{Theorem}[section]
\newtheorem*{theorem*}{Theorem}

\newtheorem*{corollary*}{Corollary}
\newtheorem{lemma}[theorem]{Lemma}
\newtheorem*{lemma*}{Lemma}
\newtheorem{proposition}[theorem]{Proposition}
\newtheorem*{proposition*}{Proposition}

\newtheorem*{conjecture*}{Conjecture}
\newtheorem{fact}[theorem]{Fact}
\theoremstyle{definition}
\newtheorem{definition}[theorem]{Definition}
\newtheorem*{definition*}{Definition}

\newtheorem*{example*}{Example}

\newtheorem*{problem*}{Problem}
\newtheorem{observation}[theorem]{Observation}

\theoremstyle{remark}

\newtheorem*{remark*}{Remark}

\title[Pattern avoidance in partial permutations]{Pattern avoidance in
  partial permutations}

\author[A. Claesson]{Anders Claesson} 
\address{The Mathematics Institute, School of Computer Science,
  Reykjavik University, Menntavegur 1, IS-101 Reykjavik, Iceland}
\email{anders@ru.is}
\thanks{A. Claesson, V. Jel\'{\i}nek and S. Kitaev were supported by the Icelandic
Research Fund, grant no. 090038011. E. Jel\'{\i}nkov\'a was supported by project
1M0021620838 of the Czech Ministry of Education. The research was conducted
while E. Jel\'{\i}nkov\'a was visiting ICE-TCS, Reykjavik University,
Iceland. }

\author[V. Jel\'{\i}nek]{V\'{\i}t Jel\'{\i}nek} 
\address{The Mathematics Institute, School of Computer Science,
  Reykjavik University, Menntavegur 1, IS-101 Reykjavik, Iceland}
\email{jelinek@kam.mff.cuni.cz}

\author[E. Jel\'{\i}nkov\'a]{Eva Jel\'{\i}nkov\'a} 
\address{Department of Applied Mathematics, Charles University in
  Prague, Malostransk\'e n\'am.  25, 118 00 Praha 1, Czech Republic}
\email{eva@kam.mff.cuni.cz}

\author[S. Kitaev]{Sergey Kitaev} 
\address{The Mathematics Institute, School of Computer Science,
  Reykjavik University, Menntavegur 1, IS-101 Reykjavik, Iceland}
\email{sergey@ru.is}

\begin{document}
\begin{abstract}
Motivated by the concept of partial words, we introduce an analogous
concept of partial permutations. A \emph{partial permutation of length
  $n$ with $k$ holes} is a sequence of symbols
$\pi=\pi_1\pi_2\dotsb\pi_n$ in which each of the symbols from the set
$\{1,2,\dotsc,n-k\}$ appears exactly once, while the remaining $k$
symbols of $\pi$ are ``holes''.

We introduce pattern-avoidance in partial permutations and prove that
most of the previous results on Wilf equivalence of permutation patterns can
be extended to partial permutations with an arbitrary number of holes. We
also show that Baxter permutations of a given length~$k$ correspond to
a Wilf-type equivalence class with respect to partial permutations
with $(k-2)$ holes. Lastly, we enumerate the partial permutations of
length $n$ with $k$ holes avoiding a given pattern of length at most
four, for each $n\ge k\ge 1$.

  \medskip
  
  \noindent{\bf MSC~(2000):} 05A15
\end{abstract}

\maketitle
\thispagestyle{empty}

\section{Introduction}\label{intro}

Let $A$ be a nonempty set, which we call an \emph{alphabet}. A \emph{word}
over $A$ is a finite sequence of elements of $A$, and the \emph{length} of
the word is the number of elements in the sequence. Assume that $\di$
is a special symbol not belonging to $A$. The symbol $\di$ will be
called \emph{a hole}.  A \emph{partial word} over
$A$ is a word over the alphabet $A\cup \{\di\}$. In the study of
partial words, the holes are usually treated as gaps that may be
filled by an arbitrary letter of~$A$. The \emph{length} of a partial
word is the number of its symbols, including the holes.

The study of partial words was initiated by Berstel and
Boasson~\cite{BB}. Partial words appear in comparing genes~\cite{Leu}; alignment
of two sequences can be viewed as a construction of two partial words
that are compatible in the sense defined in~\cite{BB}. Combinatorial
aspects of partial words that have been studied include periods in
partial words~\cite{BB,Shur}, avoidability/unavoidability of sets of partial words
\cite{BBGR,BBKPW}, squares in partial words~\cite{HHK}, and
overlap-freeness~\cite{HHKS}. For more see the book by
Blanchet-Sadri~\cite{B}.

Let $V$ be a set of symbols not containing~$\di$. A \emph{partial
  permutation of $V$} is a partial word $\pi$ such that each symbol of $V$ 
appears in $\pi$ exactly once, and all the remaining symbols of $\pi$ are 
holes. Let $\sym_n^k$ denote the set of all partial permutations of the set 
$[n-k]=\{1,2,\dotsc,n-k\}$ that have exactly $k$ holes. For example, 
$\sym_3^1$ contains the six partial permutations $12\di$, $1\di2$, $21\di$, 
$2\di1$, $\di12$, and $\di21$.  Obviously, all elements of $\sym_n^k$ have 
length~$n$, and $|\sym_n^k|=\binom{n}{k}(n-k)!=n!/k!$. Note that $\sym_n^0$ 
is the familiar symmetric group $\sym_n$. For a set $H\subseteq [n]$ of size 
$k$, we let $\sym_n^H$ denote the set of partial permutations 
$\pi_1\dotsb\pi_n\in \sym_n^k$ such that $\pi_i=\di$ if and only if $i\in H$. 
We remark that our notion of partial permutations is somewhat reminiscent of 
the notion of insertion encoding of permutations, introduced by Albert et 
al.~\cite{albert}. However, the interpretation of holes in the two settings 
is different.


In this paper, we extend the classical notion of pattern-avoiding
permutations to the more general setting of partial permutations.  Let
us first recall some definitions related to pattern avoidance in
permutations. Let $V=\{v_1,\ldots,v_n\}$ with $v_1<\dots<v_n$ be any
finite subset of $\NN$. The \emph{standardization} of a permutation
$\pi$ on $V$ is the permutation $\st(\pi)$ on $[n]$ obtained from
$\pi$ by replacing the letter $v_i$ with the letter~$i$. As an
example, $\st(19452) = 15342$. Given $p\in\sym_k$ and $\pi\in\sym_n$,
an \emph{occurrence} of $p$ in $\pi$ is a subword
$\omega=\pi_{i(1)}\cdots\pi_{i(k)}$ of $\pi$ such that
$\st(\omega)=p$; in this context $p$ is called a \emph{pattern}. If
there are no occurrences of $p$ in $\pi$ we also say that $\pi$ {\em
  avoids}~$p$. Two patterns $p$ and $q$ are called
\emph{Wilf-equivalent} if for each $n$, the number of $p$-avoiding
permutations in $\sym_n$ is equal to the number of $q$-avoiding
permutations in~$\sym_n$.

Let $\pi\in\sym_n^k$ be a partial permutation and let
$i(1)<\dots<i(n-k)$ be the indices of the non-hole elements
of~$\pi$. A permutation $\sigma\in\sym_n$ is an {\em extension} of
$\pi$ if
\[
\st(\sigma_{i(1)}\dotsb\sigma_{i(n-k)}) = \pi_{i(1)}\dotsb\pi_{i(n-k)}.
\]
For example, the partial permutation $2\di1$ has three extensions,
namely $312$, $321$ and $231$. In general, the number of extensions of
$\pi\in \sym_n^k$ is $\binom{n}{k}k!=n!/(n-k)!$.

We say that $\pi\in\sym_n^k$ \emph{avoids the pattern} $p\in
\sym_{\ell}$ if each extension of $\pi$ avoids $p$. For example,
$\pi=32\di 154$ avoids 1234, but $\pi$ does not avoid 123: the
permutation $325164$ is an extension of $\pi$ and it contains two
occurrences of 123. Let $\sym_n^k(p)$ be the set of all the partial
permutations in $\sym_n^k$ that avoid~$p$, and let
$s_n^k(p)=|\sym_n^k(p)|$. Similarly, if $H\subseteq [n]$ is a set of
indices, then $\sym_n^H(p)$ is the set of $p$-avoiding
permutations in $\sym_n^H$, and $s_n^H(p)$ is its cardinality.

We say that two patterns $p$ and $q$ are {\em $k$-Wilf-equivalent} if
$s_n^k(p)=s_n^k(q)$ for all~$n$. Notice that 0-Wilf equivalence
coincides with the standard notion of Wilf equivalence.  We also say
that two patterns $p$ and $q$ are {\em $\star$-Wilf-equivalent} if $p$
and $q$ are $k$-Wilf-equivalent for all $k\geq 0$. Two patterns $p$
and $q$ are {\em strongly $k$-Wilf-equivalent} if $s_n^H(p)=s_n^H(q)$
for each $n$ and for each $k$-element subset $H\subseteq
[n]$. Finally, $p$ and $q$ are {\em strongly $\star$-Wilf-equivalent}
if they are strongly $k$-Wilf-equivalent for all $k\geq 0$.

We note that although strong $k$-Wilf equivalence implies
$k$-Wilf-equivalence, and strong $\star$-Wilf equivalence implies
$\star$-Wilf equivalence, the converse implications are not true.
Consider for example the patterns $p=1342$
and $q=2431$. A partial permutation avoids $p$ if and only if its
reverse avoids $q$, and thus $p$ and $q$ are
$\star$-Wilf-equivalent. However, $p$ and $q$ are not strongly
1-Wilf-equivalent, and hence not strongly $\star$-Wilf-equivalent
either. To see this, we fix $H=\{2\}$ and easily check that
$s_5^H(p)=13$ while $s_5^H(q)=14$.

\subsection{Our Results}
The main goal of this paper is to establish criteria for $k$-Wilf
equivalence and $\star$-Wilf equivalence of permutation patterns. We
are able to show that most pairs of Wilf-equivalent patterns that were
discovered so far are in fact $\star$-Wilf-equivalent. The only
exception is the pair of patterns $p=2413$ and $q=1342$. Although
these patterns are known to be Wilf-equivalent~\cite{stan}, they are
neither 1-Wilf-equivalent nor 2-Wilf equivalent (see
Section~\ref{sec:length4}).

Many of our arguments rely on properties of partial 01-fillings of Ferrers diagrams. These
fillings are introduced in Section~\ref{sec:pfillings}, where we also
establish the link between partial fillings and partial permutations.

Our first main result is Theorem~\ref{thm-diag} in
Section~\ref{sec:strong-star-1}, which states that a permutation
pattern of the form $123\dotsb \ell X$ is strongly
$\star$-Wilf-equivalent to the pattern $\ell(\ell-1)\dotsb321X$, where
$X=x_{\ell+1}x_{\ell+2}\dotsb x_m$ is any permutation of
$\{\ell+1,\dotsc,m\}$.  This theorem is a strengthening of a result of
Backelin, West and Xin~\cite{bwx}, who show that patterns of this form
are Wilf-equivalent. Our proof is based on a different argument than
the original proof of Backelin, West and Xin. The main ingredient of
our proof is an involution on a set of fillings of Ferrers diagrams,
discovered by Krattenthaler~\cite{krat}. We adapt this involution to
partial fillings and use it to obtain a bijective proof of our result.

Our next main result is Theorem~\ref{thm-312} in
Section~\ref{sec:strong-star-2}, which states that for any permutation
$X$ of the set $\{4,5,\dotsc,k\}$, the two patterns $312X$ and $231X$
are strongly $\star$-Wilf-equivalent. This is also a refinement of an
earlier result involving Wilf equivalence, due to Stankova and
West~\cite{sw}. As in the previous case, the refined version requires
a different proof than the weaker version.

In Section~\ref{sec-bax}, we study the $k$-Wilf equivalence of
patterns whose length is small in terms of $k$. It is not hard to see
that all patterns of length $\ell$ are $k$-Wilf equivalent whenever
$\ell\le k+1$, because $s_n^k(p) = 0$ for every such $p$ and every~$n\ge \ell$.
Thus, the shortest patterns that exhibit nontrivial
behaviour are the patterns of length~$k+2$. For these patterns, we
show that $k$-Wilf equivalence yields a new characterization of Baxter
permutations: a pattern $p$ of length $k+2$ is a Baxter permutation if
and only if $s_n^k(p)=\binom{n}{k}$. For any non-Baxter permutation
$q$ of length $k+2$, $s_n^k(q)$ is strictly smaller than
$\binom{n}{k}$ and is in fact a polynomial in~$n$ of degree at
most~$k-1$.

In Section~\ref{sec:length4}, we focus on explicit enumeration of
$s_n^k(p)$ for small patterns~$p$. We obtain explicit closed-form
formulas for $s_n^k(p)$ for every $p$ of length at most four and every
$k\ge 1$.

\subsection{A note on monotone patterns}\label{ssec:mono}

Before we present our main results, let us illustrate the above definitions
on the example of the monotone pattern $12\dotsb\ell$.  Let
$\pi\in\sym_n^k$, and let $\pi'\in\sym_{n-k}$ be the permutation
obtained from $\pi$ by deleting all the holes. Note that $\pi$ avoids the
pattern $12\dotsb\ell$ if and only if $\pi'$ avoids
$12\dotsb(\ell-k)$. Thus,
\begin{equation}
s_n^k(12\dotsb \ell)=\binom{n}{k}s_n^0(12\dotsb(\ell-k)),\label{eq:monotone}
\end{equation}
where $\binom{n}{k}$ counts the possibilities of placing $k$ holes.
For instance, if $\ell=k+3$ then $s_n^k(12\dotsb
\ell)=\binom{n}{k}s_n^0(123)$, and it is well known that
$s_n^0(123)=C_n$, the $n$-th Catalan number. For general $\ell$,
Regev~\cite{Reg} found an asymptotic formula for $s_n^0(12\dotsb
\ell)$, which can be used to obtain an asymptotic formula for $s_n^k(12\dotsb
\ell)$ as $n$ tends to infinity.

\section{Partial fillings}\label{sec:pfillings}

In this section, we introduce the necessary definitions related to
partial fillings of Ferrers diagrams. These notions will later be useful
in our proofs of $\star$-Wilf equivalence of patterns.

Let $\lambda=(\lambda_1\ge\lambda_2\ge\dotsb\ge\lambda_k)$ be a
non-increasing sequence of $k$ nonnegative integers. A \emph{Ferrers
  diagram with shape $\lambda$} is a bottom-justified array $D$ of
cells arranged into $k$ columns, such that the $j$-th column from the
left has exactly $\lambda_j$ cells. Note that our definition of
Ferrers diagram is slightly more general than usual, in that we allow
columns with no cells.  If each column of $D$ has at least one cell,
then we call $D$ a \emph{proper Ferrers diagram}. Note that every
row of a Ferrers diagram $D$ has nonzero length (while we allow
columns of zero height). If all the columns of $D$ have zero
height---in other words, $D$ has no rows---then $D$ is called
\emph{degenerate}.

For the sake of consistency, we assume throughout this paper that the
rows of each diagram and each matrix are numbered from bottom to top,
with the bottom row having number 1. Similarly, the columns are
numbered from left to right, with column 1 being the leftmost column.

By \emph{cell $(i,j)$} of a Ferrers diagram $D$ we mean the cell of $D$ that 
is the intersection of $i$-th row and $j$-th column of the diagram. We assume 
that the cell $(i,j)$ is a unit square whose corners are lattice points with 
coordinates $(i-1,j-1)$, $(i,j-1)$, $(i-1,j)$ and $(i,j)$. The point $(0,0)$ 
is the bottom-left corner of the whole diagram. We say a diagram $D$ 
\emph{contains} a lattice point $(i,j)$ if either $j=0$ and the first column 
of $D$ has height at least $i$, or $j>0$ and the $j$-th column of $D$ has 
height at least~$i$. A point $(i,j)$ is a \emph{boundary point} of the 
diagram $D$ if $D$ contains the point $(i,j)$ but does not contain the cell 
$(i+1,j+1)$ (see Figure~\ref{fig-diag}). Note that a Ferrers diagram with $r$ 
rows and $c$ columns has $r+c+1$ boundary points.

\begin{figure}[ht]
\includegraphics{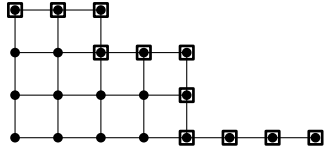}
\caption{A Ferrers diagram with shape $(3,3,2,2,0,0,0)$. The black dots represent the
  points. The black dots in squares are the boundary
  points.}\label{fig-diag}
\end{figure}


A \emph{01-filling} of a Ferrers diagram assigns to each cell the value 0 or 
1. A 01-filling is called a \emph{transversal filling} (or just a 
\emph{transversal}) if each row and each column has exactly one 1-cell.  A 
01-filling is \emph{sparse} if each row and each column has at most one 
1-cell.  A permutation $p=p_1p_2\dotsb p_\ell\in\sym_\ell$ can be represented 
by a \emph{permutation matrix} which is a 01-matrix of size $\ell\times 
\ell$, whose cell $(i,j)$ is equal to 1 if and only if $p_j=i$. If there is 
no risk of confusion, we abuse terminology by identifying a permutation 
pattern $p$ with the corresponding permutation matrix. Note that a 
permutation matrix is a transversal of a diagram with square shape.

Let $P$ be permutation matrix of size $n\times n$, and let $F$ be a sparse
filling of a Ferrers diagram. We say that $F$
\emph{contains} $P$ if $F$ has a (not necessarily contiguous) square
submatrix of size $n\times n$ which induces in $F$ a subfilling equal
to~$P$. This notion of containment generalizes usual permutation
containment.

We now extend the notion of partial permutations to partial fillings
of diagrams. 
Let $D$ be a Ferrers diagram with $k$ columns. Let $H$ be a subset of
the set of columns of~$D$. Let $\phi$ be a function that assigns to every cell of
$D$ one of the three symbols $0$, $1$ and \di, such that every cell in
a column belonging to $H$ is filled with \di, while every cell in a
column not belonging to $H$ is filled with $0$ or $1$.  The pair
$F=(\phi,H)$, will be referred to as \emph{a partial 01-filling (or
a partial filling) of the diagram $D$}. See Figure~\ref{fig-pfill}. The
columns from the set $H$ will be called \emph{the \di-columns} of $F$,
while the remaining columns will be called \emph{the standard
  columns}. Observe that if the diagram $D$ has columns of height zero, then
$\phi$ itself is not sufficient to determine the filling $F$, because
it does not allow us to determine whether the zero-height columns are
\di-columns or standard columns. For our purposes, it is necessary to
distinguish between partial fillings that differ only by the status of
their zero-height columns.

\begin{figure}[ht]
\includegraphics{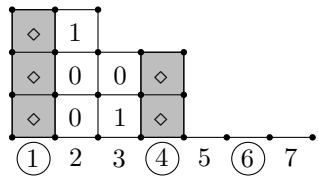}
\caption{A partial filling with \di-columns $1$, $4$ and $6$.}\label{fig-pfill}
\end{figure}

We say that a partial 01-filling is a \emph{partial transversal filling} (or 
simply a \emph{partial transversal}) if every row and every standard column 
has exactly one 1-cell. We say that a partial 01-filling is \emph{sparse} if 
every row and every standard column has at most one 1-cell.  A \emph{partial 
01-matrix} is a partial filling of a (possibly degenerate) rectangular 
diagram. In this paper, we only deal with transversal and sparse partial fillings.

There is a natural correspondence between partial permutations and
transversal partial 01-matrices. Let $\pi\in \sym_n^k$ be a partial
permutation. A \emph{partial permutation matrix representing $\pi$} is
a partial 01-matrix $P$ with $n-k$ rows and $n$ columns, with the
following properties:
\begin{itemize}
\item If $\pi_j=\di$, then the $j$-th column of
  $P$ is a \di-column.
\item If $\pi_j$ is equal to a number $i$, then the $j$-th
  column is a standard column. Also, the cell in column $j$ and row $i$
  is filled with $1$, and the remaining cells in column $j$ are filled
  with $0$'s.
\end{itemize}  
If there is no risk of confusion, we will make no distinction between a 
partial permutation and the corresponding partial permutation matrix.

To define pattern-avoidance for partial fillings, we first introduce the
concept of substitution into a \di-column, which is analogous to
substituting a number for a hole in a partial permutation. The idea is
to insert a new row with a 1-cell in the \di-column; this increases the height
of the diagram by one. Let us now describe the substitution formally.

Let $F$ be a partial filling of a Ferrers diagram with $m$ columns. Assume
that the $j$-th column of $F$ is a \di-column. Let $h$ be the height
of the $j$-th column. A \emph{substitution} into the $j$-th column is
an operation consisting of the following steps:
\begin{enumerate}
\item Choose a number $i$, with $1\le i\le h+1$.
\item Insert a new row into the diagram, between rows $i-1$ and
  $i$. The newly inserted row must not be longer than the $(i-1)$-th
  row, and it must not be shorter than the $i$-th row, so that the new
  diagram is still a Ferrers diagram. If $i=1$, we also assume that the new row
  has length $m$, so that the number of columns does not
  increase during the substitution.
\item Fill all the cells in column $j$ with the symbol $0$, except for
  the cell in the newly inserted row, which is filled with $1$. Remove
  column $j$ from the set of \di-columns.
\item Fill all the remaining cells of the new row with $0$ if they
  belong to a standard column, and with \di{} if they belong to a
  \di-column.
\end{enumerate}

Figure~\ref{fig-subst} illustrates an example of substitution.

\begin{figure}[ht]
\includegraphics{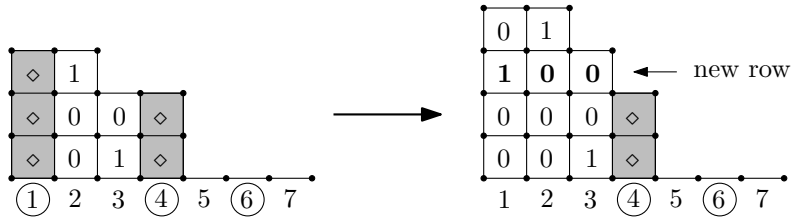}
\caption{A substitution into the first column of a partial filling, involving
  an insertion of a new row between the second and third rows of the
  original partial filling.}\label{fig-subst}
\end{figure}

Note that a substitution into a partial filling increases the number of rows
by~1. A substitution into a transversal (resp. sparse) partial filling
produces a new transversal (resp. sparse) partial filling. A partial filling $F$
with $k$ \di-columns can be transformed into a (non-partial)
filling $F'$ by a sequence of $k$ substitutions; we then say
that $F'$ is an \emph{extension} of~$F$.

Let $P$ be a permutation matrix. We say that a partial filling $F$ of a
Ferrers diagram \emph{avoids} $P$ if every extension of $F$
avoids~$P$. Note that a partial permutation $\pi\in S^n_k$ avoids a
permutation $p$, if and only if the partial permutation matrix
representing $\pi$ avoids the permutation matrix representing~$p$.

\section{A generalization of a Wilf-equivalence by Babson and West}

We say that two permutation matrices $P$ and $Q$ are 
\emph{shape-$\star$-Wilf-equivalent}, if for every Ferrers diagram $D$ there 
is a bijection between $P$-avoiding and $Q$-avoiding partial transversals of 
$D$ that preserves the set of \di-columns. Observe that if two permutations 
are shape-$\star$-Wilf-equivalent, then they are also strongly 
$\star$-Wilf-equivalent, because a partial permutation matrix is a special 
case of a partial filling of a Ferrers diagram.

The notion of shape-$\star$-Wilf-equivalence is motivated by the following
proposition, which extends an analogous result of Babson and West~\cite{BW}
for shape-Wilf-equivalence of non-partial permutations.

\begin{proposition}\label{prop-shape} 
  Let $P$ and $Q$ be shape-$\star$-Wilf-equivalent permutations, let $X$ be
  an arbitrary permutation. Then the two permutations $\diag{P}{X}$
  and $\diag{Q}{X}$ are strongly $\star$-Wilf-equivalent.
\end{proposition}

Our proof of Proposition~\ref{prop-shape} is based on the same idea as the original argument of Babson and West~\cite{BW}. 
Before we state the proof, we need some
preparation. Let $M$ be a partial matrix with $r$ rows and $c$
columns. Let $i$ and $j$ be numbers satisfying $0\le i\le r$ and $0\le
j\le c$. Let $M(>\!i,>\!j)$ be the submatrix of $M$ formed by the cells
$(i',j')$ satisfying $i'>i$ and $j'>j$. In other words, $M(>\!i,>\!j)$
is formed by the cells to the right and above the
point $(i,j)$. The matrix $M(>\!r,>\!j)$ is assumed to be the
degenerate matrix with $0$ rows and $c-j$ columns, while $M(>\!i,>\!c)$ is 
assumed to be the empty matrix for any value of~$i$.
When the matrix $M(>\!i,>\!j)$ intersects a $\di$-column of $M$,
we assume that the column is also a $\di$-column of~$M(>\!i,>\!j)$, 
and similarly for standard columns.
                            
We will also use the analogous notation $M(\le\!i, \le\!j)$ to denote the submatrix of 
$M$ formed by the cells to the left and below the point~$(i,j)$.

Note that if $M$ is a partial permutation matrix, then $M(>\!i,>\!j)$ and 
$M(\le\!i, \le\!j)$ are sparse partial matrices. 

Let $X$ be any nonempty permutation matrix, and $M$ be a partial
permutation matrix. We say that a point $(i,j)$ of $M$ is
\emph{dominated by $X$ in $M$} if the partial matrix $M(>\!i,>\!j)$
contains~$X$. Similarly, we say that a cell of $M$ is dominated by
$X$, if the top-right corner of the cell is dominated by~$X$. Note
that if a point $(i,j)$ is dominated by $X$ in $M$, then all the cells
and points in $M(\le\!i,\le\!j)$ are dominated by $X$ as well. In particular, the
points dominated by $X$ form a (not necessarily proper) Ferrers diagram.

Let $k\equiv k(M)\ge 0$ be the largest integer such that the point $(0,k)$ is dominated by~$X$. If no such integer exists, set $k=0$.
Observe that all the cells of $M$ dominated by $X$ appear in the
leftmost $k$ columns of~$M$. Let $M(X)$ be the partial subfilling of
$M$ induced by the points dominated by $X$; formally $M(X)$ is defined
as follows:
\begin{itemize}
\item $M(X)$ has $k$ columns, some of which might have height zero,
\item the cells of $M(X)$ are exactly the cells of $M$ dominated by $X$,
\item a column $j$ of $M(X)$ is a \di-column, if and only if $j$
  is a \di-column of $M$.
\end{itemize}

Our proof of Proposition~\ref{prop-shape} is based on the next lemma.

\begin{lemma}\label{lem-shape} 
  Let $M$ be a partial permutation matrix, and let $P$ and $X$ be
  permutation matrices. Then $M$ contains $\diag{P}{X}$ if and only if
  $M(X)$ contains~$P$.
\end{lemma}
\begin{proof}
  Assume that $M$ contains $\diag{P}{X}$. It is easy to see that $M$
  must then contain a point $(i,j)$ such that the matrix $M(>\!i,>\!j)$
  contains $X$ while the matrix $M(\le\!i, \le\!j)$ contains $P$. By
  definition, the point $(i,j)$ is dominated by $X$ in $M$, and hence
  all the points of $M(\le\!i, \le\!j)$ are dominated by $X$ as
  well. Thus, $M(\le\!i,\le\!j)$ is a (possibly degenerate) submatrix of
  $M(X)$, which implies that $M(X)$ contains~$P$.
  
  The converse implication is proved by an analogous argument.
\end{proof}

We are now ready to prove Proposition~\ref{prop-shape}.

\begin{proof}[Proof of Proposition~\ref{prop-shape}]
  Let $P$ and $Q$ be two shape-$\star$-Wilf-equivalent matrices, and let $f$
  be the bijection that maps $P$-avoiding partial transversals to
  $Q$-avoiding partial transversals of the same diagram and with the
  same \di-columns.  Let $M$ be a partial permutation matrix
  avoiding $\diag{P}{X}$.

  By Lemma~\ref{lem-shape}, $M(X)$ is a sparse partial filling avoiding~$P$.
  Let $F$ denote the partial filling $M(X)$. Consider the transversal
  partial filling $F^-$ obtained from $F$ by removing all the rows and all
  the standard columns that contain no 1-cell. Clearly $F^-$ is a $P$-avoiding partial transversal. 
  Use the bijection $f$
  to map the partial filling $F^-$ to a $Q$-avoiding partial transversal $G^-$ of the
  same shape as~$F^-$. By reinserting the zero rows and zero standard
  columns into $G^-$, we obtain a sparse $Q$-avoiding filling $G$ of
  the same shape as $F$. Let us transform the partial matrix $M$ into
  a partial matrix $N$ by replacing the cells of $M(X)$ with the cells
  of $G$, while the values of all the remaining cells of $M$ remain the
  same.

  We claim that the matrix $N$ avoids $\diag{Q}{X}$. By
  Lemma~\ref{lem-shape}, this is equivalent to claiming that $N(X)$
  avoids $Q$. We will in fact show that $N(X)$ is exactly the filling~$G$. To show this, it is enough to show, for any point $(i,j)$, that
  $M(X)$ contains $(i,j)$ if and only if $N(X)$ contains~$(i,j)$.
  This will imply that $M(X)$ and $N(X)$ have the same shape, and hence $G=N(X)$. 

  Let $(i,j)$ be a point of $M$
  not belonging to~$M(X)$. Since $(i,j)$ is not in $M(X)$, we see
  that $M(>\!i,>\!j)$ is the same matrix as $N(>\!i,>\!j)$, and this
  means that $(i,j)$ is not dominated by $X$ in~$N$, hence $(i,j)$ is not in $N(X)$. 

  Now assume that
  $(i,j)$ is a point of $M(X)$. Let $(i',j')$ be a boundary point of $M(X)$ such that $i'\ge i$ and $j'\ge j$.    
  Then the
  matrix $M(>\!i',>\!j')$ is equal to the matrix $N(>\!i',>\!j')$,
  showing that $(i',j')$ belongs to $N(X)$, and 
  hence $(i,j)$ belongs to $N(X)$ as well. We conclude
  that $N(X)$ and $M(X)$ have the same
  shape. This means that $N(X)$ avoids $Q$, and hence $N$ avoids $\diag{Q}{X}$.

  Since we have shown that $M(X)$ and $N(X)$ have the same shape, it
  is also easy to see that the above-described transformation
  $M\mapsto N$ can be inverted, showing that the transformation is a
  bijection between partial permutation matrices avoiding
  $\diag{P}{X}$ and those avoiding $\diag{Q}{X}$. The bijection clearly preserves the position of $\di$-columns, and shows that $\diag{P}{X}$ and $\diag{Q}{X}$ are strongly $\star$-Wilf equivalent.
\end{proof}

\section{Strong $\star$-Wilf-equivalence of 
  $12\cdots \ell X$ and $\ell(\ell-1)\cdots1X$}\label{sec:strong-star-1}

We will use Proposition~\ref{prop-shape} as the main tool to prove
strong $\star$-Wilf equivalence. To apply the proposition, we need to
find pairs of shape-$\star$-Wilf-equivalent patterns. A family of such
pairs is provided by the next proposition, which extends previous
results of Backelin, West and Xin~\cite{bwx}.

\begin{proposition}\label{prop-diag}
  Let $I_\ell=12\dotsb\ell$ be the identity permutation of order
  $\ell$, and let $J_\ell=\ell(\ell-1)\dotsb 21$ be the anti-identity
  permutation of order~$\ell$. The permutations $I_\ell$ and $J_\ell$
  are shape-$\star$-Wilf-equivalent.
\end{proposition}

Before stating the proof, we introduce some notation and
terminology.
Let $F$ be a sparse partial filling of a Ferrers diagram, and let $(i,j)$ be
a boundary point of $F$. Let $h(F,j)$ denote the number of
\di-columns among the first $j$ columns of~$F$. Let $I(F,i,j)$ denote
the largest integer $\ell$ such that the partial matrix $F(\le\!i,
\le\!j)$ contains $I_\ell$. Similarly, let $J(F,i,j)$ denote the largest
$\ell$ such that $F(\le\!i, \le\!j)$ contains~$J_\ell$.

We let $F_0$ denote the (non-partial) sparse filling obtained by
replacing all the symbols \di{} in $F$ by zeros.

Let us state without proof the following simple observation.

\begin{observation}\label{obs-ij} Let $F$ be a sparse partial filling.
  \begin{enumerate}
  \item $F$ contains a permutation matrix $P$ if and only if $F$ has a
    boundary point $(i,j)$ such that $F(\le i, \le j)$ contains $P$.
  \item For any boundary point $(i,j)$, we have
    $I(F,i,j)=h(F,j)+I(F_0,i,j)$ and $J(F,i,j)=h(F,j)+J(F_0,i,j)$.
\end{enumerate}
\end{observation}

The key to the proof of Proposition~\ref{prop-diag} is the following
result, which is a direct consequence of more general results of
Krattenthaler~\cite[Theorems 1--3]{krat} obtained using the theory of growth
diagrams.

\begin{fact}\label{fac-krat} 
  Let $D$ be a Ferrers diagram. There is a bijective mapping $\kappa$
  from the set of all (non-partial) sparse fillings of $D$ onto itself, with the
  following properties.
\begin{enumerate}
\item For any boundary point $(i,j)$ of $D$, and for any sparse
  filling $F$, we have $I(F,i,j)=J(\kappa(F),i,j)$ and
  $J(F,i,j)=I(\kappa(F),i,j)$.
\item The mapping $\kappa$ preserves the number of 1-cells in each row
  and column. In other words, if a row (or column) of a sparse filling
  $F$ has no 1-cell, then the same row (or column) of
  $\kappa(F)$ has no 1-cell either.
\end{enumerate}
\end{fact}

In Krattenthaler's paper, the results are stated in terms of proper
Ferrers diagrams. However, the bijection obviously extends to Ferrers
diagrams with zero-height columns as well. This is because adding
zero-height columns to a (non-partial) filling does not affect pattern
containment.

From the previous theorem, we easily obtain the proof of the main
proposition in this section.

\begin{proof}[Proof of Proposition~\ref{prop-diag}] 
  Let $D$ be a Ferrers diagram. Let $F$ be an
  $I_\ell$-avoiding partial transversal of $D$. Let $F_0$ be the sparse
  filling obtained by replacing all the \di{} symbols of $F$ by
  zeros. Define $G_0=\kappa(F_0)$, where $\kappa$ is the bijection
  from Fact~\ref{fac-krat}. Note that all the \di-columns of
  $F$ are filled with zeros both in $F_0$ and $G_0$. Let $G$ be the
  sparse partial filling obtained from $G_0$ by replacing zeros with \di\ in
  all such columns. Then $G$ is a sparse partial filling with the same set of
  \di-columns as $F$.

  We see that for any boundary point $(i,j)$ of the diagram $D$,
  $h(F,j)=h(G,j)$. By the properties of $\kappa$, we further
  obtain $I(F_0,i,j)=J(G_0,i,j)$. In view of Observation~\ref{obs-ij},
  this implies that $G$ is a $J_\ell$-avoiding filling. It is clear that
  this construction can be inverted, thus giving the required
  bijection between $I_\ell$-avoiding and $J_\ell$-avoiding transversal
  partial fillings of~$D$.
\end{proof}

Combining Proposition~\ref{prop-shape} with
Proposition~\ref{prop-diag}, we get directly the main result of this
section.

\begin{theorem}\label{thm-diag}
  For any $\ell\leq m$, and for any permutation
  $X$ of $\{\ell+1,\dotsc,m\}$,
  the permutation pattern $123\dotsb(\ell-1)\ell X$ is
  strongly $\star$-Wilf-equivalent to the pattern
  $\ell(\ell-1)\dotsb 21X$.
\end{theorem}         

Notice that Theorem~\ref{thm-diag} implies, among other things, that all the 
patterns of size three are strongly $\star$-Wilf-equivalent.

\section{Strong $\star$-Wilf-equivalence of $312X$ and $231X$}
\label{sec:strong-star-2}

We will now focus on the two patterns 312 and 231. The main result of
this section is the following theorem.

\begin{theorem}\label{thm-312}
  The patterns 312 and 231 are shape-$\star$-Wilf-equivalent. By
  Proposition~\ref{prop-shape}, this means that for any permutation
  $X$ of the set $\{4,5,\dotsc,m\}$, the two permutations $312X$ and
  $231X$ are strongly $\star$-Wilf-equivalent.
\end{theorem}

Theorem~\ref{thm-312} generalizes a result of Stankova and
West~\cite{sw}, who have shown that 312 and 231 are shape-Wilf
equivalent. The original proof of Stankova and West~\cite{sw} is
rather complicated, and does not seem to admit a straightforward
generalization to the setting of shape-$\star$-Wilf-equivalence. Our proof
of Theorem~\ref{thm-312} is different from the argument of Stankova
and West, and it is based on a bijection of Jel\'\i nek~\cite{jel},
obtained in the context of pattern-avoiding ordered matchings.

Let us begin by giving a description of 312-avoiding and 231-avoiding
partial transversals. We first introduce some
terminology. Let $D$ be a Ferrers diagram with a prescribed set of
\di-columns. If $j$ is the index of the leftmost \di-column
of $D$, we say that the columns $1,2,\dotsc,j-1$ form \emph{the left
  part of $D$}, and the columns to the right of column $j$ form
\emph{the right part of $D$}. We also say that the rows that intersect
column $j$ form the \emph{bottom part of $D$} and the remaining rows
form \emph{the top part of~$D$}. See Figure~\ref{fig-leftright}.

If $D$ has no \di-column, then the left part and the top part is
the whole diagram $D$, while the right part and the bottom part are
empty.

The intersection of the left part and the top part of $D$ will be
referred to as \emph{the top-left part of~$D$}. The top-right,
bottom-left and bottom-right parts are defined analogously. Note that
the top-right part contains no cells of~$D$, the top-left and bottom-right
parts form a Ferrers subdiagram of $D$, and the bottom-left part is a
rectangle.

\begin{figure}
\includegraphics{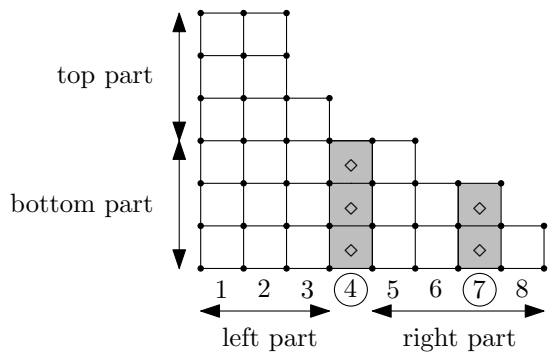}
\caption{An example of a Ferrers diagram with two \di-columns. The
  left, right, top, and bottom parts are shown.}\label{fig-leftright}
\end{figure}

\begin{observation}\label{obs-312} A partial transversal $F$ of a Ferrers diagram avoids
the pattern 312 if and only if it satisfies the following conditions:
\begin{enumerate}
\item[(C1)] $F$ has at most two \di-columns.
\item[(C2)] If $F$ has at least three columns, then at most one
  \di-column of $F$ has nonzero height.
\item[(C3)] Let $i<i'$ be a pair of rows, let $j<j'$ be a pair of
  columns. If the row $i'$ intersects column $j'$ inside $F$, and if
  the $2\times 2$ submatrix of $F$ induced by rows $i,i'$ and columns
  $j,j'$ is equal to the matrix
  $\left(\begin{smallmatrix}1&0\\0&1\end{smallmatrix}\right)$, then
    either the two columns $j,j'$ both belong to the left part, or
    they both belong to the right part (in other words, the
    configuration depicted in Figure~\ref{fig-c3} is forbidden).
\item[(C4)] The subfilling induced by the left part of $F$ avoids $312$.
\item[(C5)] The subfilling induced by the right part of $F$ avoids $12$.
\item[(C6)] The subfilling induced by bottom-left part of $F$ avoids $21$.
\end{enumerate}
\end{observation}

\begin{figure}
\includegraphics{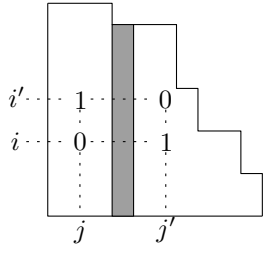}
\caption{The configuration forbidden by condition (C3) of
  Observation~\ref{obs-312}. The column $j$ is in the left part of the
  diagram, while $j'$ is in the right part.}\label{fig-c3}
\end{figure}

\begin{proof}
  It is easy to see that if any of the six conditions fails, then $F$ contains
  the pattern 312.

  To prove the converse, assume that $F$ has an occurrence of the pattern 312 
  that intersects three columns $j<j'<j''$. Choose the occurrence of 312 in 
  such a way that among the three columns $j,j'$ and $j''$, there are as 
  many \di-columns as possible.
 
  If all the three columns $j,j',j''$ are \di-columns, (C1)
  fails.  If two of the three columns are \di-columns, (C2)
  fails. If $j$ is a \di-column and $j'$ and $j''$ are standard,
  (C5) fails.  

Assume $j$ and $j''$ are standard columns and $j'$ is a
  \di-column. If $j'$ is the leftmost $\di$-column, (C3) fails, 
otherwise (C2) fails. Assume $j''$ is a \di-column and $j$ and $j'$ are
  standard. If $j''$ is the leftmost \di-column, (C6) fails,
  otherwise (C2) fails.

  Assume all the three columns are standard. Let $i<i'<i''$ be the three rows 
  that are intersected by the chosen occurrence of 312. If there is a 
  \di-column that intersects all the three rows $i,i',i''$, we may find an 
  occurrence of 312 that uses this \di-column, contradicting our choice of 
  $j,j'$ and $j''$.  On the other hand, if no \di-column intersects the three 
  rows, then the whole submatrix inducing 312 is in the left part and (C4) 
  fails.
\end{proof}

Next, we state a similar description of 231-avoiding partial transversals.

\begin{observation}\label{obs-231} 
  A partial transversal $F$ of a Ferrers diagram avoids the pattern 231 if and only
  if it satisfies the following conditions (the first three
  conditions are the same as the corresponding three conditions of
  Observation~\ref{obs-312}):
  \begin{enumerate}
  \item[(C1')] $F$ has at most two \di-columns.
  \item[(C2')] If $F$ has at least three columns, then at most one
    \di-column of $F$ has nonzero height.
  \item[(C3')] Let $i<i'$ be a pair of rows, let $j<j'$ be a pair of
    columns. If the row $i'$ intersects column $j'$ inside $F$, and if
    the $2\times 2$ submatrix of $F$ induced by rows $i,i'$ and
    columns $j,j'$ is equal to the matrix
    $\left(\begin{smallmatrix}1&0\\0&1\end{smallmatrix}\right)$, then
      either the two columns $j,j'$ both belong to the left part, or
      they both belong to the right part.
  \item[(C4')] The subfilling induced by the left part of $F$ avoids $231$.
  \item[(C5')] The subfilling induced by the right part of $F$ avoids $21$.
  \item[(C6')] The subfilling induced by bottom-left part of $F$ avoids $12$.
  \end{enumerate}
\end{observation}

The proof of Observation~\ref{obs-231} is analogous to the proof of
Observation~\ref{obs-312}, and we omit it.

In the next part of our argument, we will look in more detail at
fillings satisfying some of the Conditions (C1) to (C6), or some of
the Conditions (C1') to (C6').

For later reference, we state explicitly the following easy facts
about transversal fillings of Ferrers diagrams that avoid permutation
matrices of size~2 (see, e.g.,~\cite{BW}).

\begin{fact}\label{fac-12} 
Assume that $D$ is a Ferrers diagram that has at least one (non-partial) 
transversal. The following holds.
\begin{itemize}
\item 
The diagram $D$ has exactly one 12-avoiding transversal. To construct this 
transversal, take the rows of $D$ in top-to-bottom order, and in each row $i$, 
insert a 1-cell into the leftmost column that has no 1-cell in any of the 
rows above row~$i$.
\item 
The diagram $D$ has exactly one 21-avoiding transversal. To construct this 
transversal, take the rows of $D$ in top-to-bottom order, and in each row $i$, 
insert a 1-cell into the rightmost column that has no 1-cell in any of the 
rows above row~$i$.
\end{itemize}
\end{fact}

Our next goal is to give a more convenient description of the
partial fillings that satisfy Conditions (C1), (C2) and (C3) (which are equal
to (C1'), (C2') and (C3'), respectively). Let $D$ be a Ferrers diagram
with a prescribed set of \di-columns, and with $k$ rows in its
bottom part. We will distinguish two types of rows of $D$, which we
refer to as \emph{rightist rows} and \emph{leftist rows} (see
Figure~\ref{fig-rightist}). The rightist rows are defined inductively
as follows. None of the rows in the top part is rightist. The $k$-th
row (i.e., the highest row in the bottom part) is rightist if and only
if it has at least one cell in the right part of $D$. For any $i<k$,
the $i$-th row is rightist if and only if the number of cells in the
$i$-th row belonging to the right part of $D$ is greater than the
number of rightist rows that are above row~$i$. A row is leftist if it
is not rightist.

\begin{figure}
\includegraphics[scale=0.7]{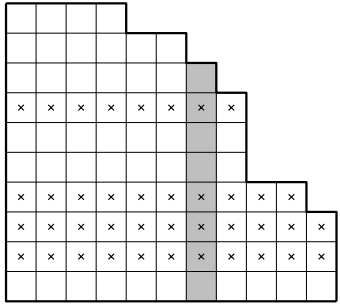}
\caption{A Ferrers diagram with one \di-column, indicated in
  gray. The rows with crosses are the rightist rows
  of~$D$.}\label{fig-rightist}
\end{figure}

The distinction between leftist and rightist rows is motivated by the
following lemma.

\begin{lemma}\label{lem-leftist} 
  Let $D$ be a Ferrers diagram, and let $F$ be a partial transversal of $D$ that
  satisfies (C1) and (C2). The following statements are equivalent.
  \begin{enumerate}
  \item[(a)] $F$ satisfies (C3).
  \item[(b)] All the 1-cells in the leftist rows appear in the left part of~$F$.
  \item[(c)] All the 1-cells in the rightist rows appear in the right part of~$F$.
  \end{enumerate}
\end{lemma}
\begin{proof}
  Let us first argue that the statements (b) and (c) are
  equivalent. To see this, notice first that in all the partial transversals of
  $D$, the number of 1-cells in the right part is the same, since each non-degenerate
  column in the right part has exactly one 1-cell. Consequently, all
  the partial transversals of $D$ also have the same number of 1-cells in the
  bottom-left part, because the number of 1-cells in the bottom-left
  part is equal to the number of bottom rows minus the number of
  non-degenerate right columns.

  We claim that the number of rightist rows is equal to the number of non-degenerate
  columns in the right part. To see this, consider the (unique)
  partial transversal $F_{21}$ of $D$ in which no two standard columns contain
  the pattern~$21$. The characterization of Fact~\ref{fac-12} easily
  implies that in $F_{21}$, a row has a 1-cell in the right part, if
  and only if it is a rightist row. Thus, in the partial filling $F_{21}$,
  and hence in any other partial transversal of $D$, the number of rightist rows
  is equal to the number of 1-cells in the right part of~$D$, which is
  equal to the number of non-degenerate right columns.

  Thus, if in a partial transversal $F$ there is a leftist row that has a 1-cell
  in the right part of $D$, there must also be a rightist row with a
  1-cell in the left part of~$D$, and vice versa. In other words,
  conditions (b) and (c) are indeed equivalent for any partial transversal~$F$.

  Assume now that $F$ is a partial transversal that satisfies (a). We claim that
  $F$ satisfies (c) as well. For contradiction, assume that there is a
  rightist row $i$ that contains a 1-cell in the left part of~$F$. Choose
  $i$ as large as possible. Let $j$ be the column
  containing the 1-cell in row~$i$. See Figure~\ref{fig-lem10}.

\begin{figure}
\includegraphics[scale=0.7]{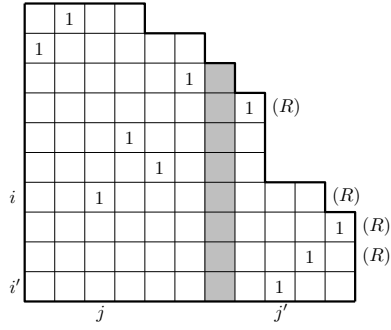}
\caption{An example of a partial transversal violating condition (c) of 
Lemma~\ref{lem-leftist}. Rightist rows are marked by $(R)$.}\label{fig-lem10}
\end{figure}

  Since $i$ is a rightist row, it follows that the number of cells in
  the right part of $i$ is greater than the number of rightist rows
  above~$i$. We may thus find a column $j'$ in the right part of $D$
  that intersects row $i$ and whose 1-cell does not belong to any of
  the rightist rows above row~$i$. Let $i'$ be the row that contains
  the 1-cell in column~$j'$. If $i'<i$, then the two rows $i,i'$ and
  the two columns $j,j'$ induce the pattern that was forbidden by
  (C3), which contradicts statement~(a).

  Thus, we see that $i'>i$. By the choice of $j'$, this implies that
  $i'$ is a leftist row. Furthermore, by the choice of $i$, we know
  that all the rightist rows above $i$, and hence all the rightist
  rows above $i'$, have a 1-cell in the right part. Since row $i'$ has
  a 1-cell in the right part as well, it means that the number of
  cells in the right part of row $i'$ is greater than the number of
  rightist rows above row~$i'$. This contradicts the fact that $i'$ is
  a leftist row. This contradiction proves that (a) implies~(c).

  It remains to show that statement (c) implies statement~(a). Assume $F$ is 
  a partial transversal that satisfies (c), and hence also~(b). For 
  contradiction, assume that $F$ contains the pattern forbidden by 
  statement~(a). Assume that the forbidden pattern is induced by a pair of 
  rows $i<i'$ and a pair of columns $j<j'$, where the column $j'$ is in the 
  right part and the column $j$ in the left part, and the two cells $(i',j)$ 
  and $(i,j')$ are 1-cells, as in Figure~\ref{fig-c3}. 

  By statement (c), the row $i'$ must be
  leftist, since it has a 1-cell in the left part. However, the number
  of cells in the right part of row $i'$ must be greater than the
  number of rightist rows above row $i'$, because all the rightist
  rows above row $i'$ have 1-cells in distinct right columns
  intersecting row $i'$, and all these columns must be different from
  column $j'$, whose 1-cell is in row $i$ below row~$i'$. This
  contradicts the fact that $i'$ is a leftist row, and completes the
  proof of the lemma.
\end{proof}

Lemma~\ref{lem-leftist}, together with Observations~\ref{obs-312}
and~\ref{obs-231}, shows that in any partial transversal avoiding 312 or 231,
each 1-cell is either in the intersection of a rightist row with a
right column, or the intersection of a leftist row and a left column.

The next lemma provides the main ingredient of our proof of
Theorem~\ref{thm-312}.

\begin{lemma}[Key Lemma]\label{lem-key312} 
  Let $k\ge 1$ be an integer, and let $D$ be a proper Ferrers diagram
  with the property that the bottom $k$ rows of $D$ all have the same
  length. Let $\cFk(D,312,21)$ be the set of all (non-partial) transversals
  of $D$ that avoid 312 and have the additional property that their
  bottom $k$ rows avoid 21. Let $\cFk(D,231,12)$ be the set of all (non-partial)
  transversals of $D$ that avoid 231 and have the additional
  property that their bottom $k$ rows avoid 12. Then
  $|\cFk(D,312,21)|=|\cFk(D,231,12)|$.
\end{lemma}

Before we prove the Key Lemma, let us explain how it implies
Theorem~\ref{thm-312}.

\begin{proof}[Proof of Theorem~\ref{thm-312} from Lemma~\ref{lem-key312}]
  Let $D$ be a Ferrers diagram with a prescribed set of
  \di-columns. Assume that $D$ has at least one partial transversal. Our goal is
  to show that the number of 312-avoiding partial transversals of $D$ is equal to
  the number of its 231-avoiding partial transversals.

  Assume that $D$ satisfies conditions (C1) and (C2), otherwise it has
  no 312-avoiding or 231-avoiding partial transversal. Let $k$ be the number of
  leftist rows in the bottom part of~$D$. Let $D_L$ be the subdiagram
  of $D$ formed by the cells that are intersections of leftist rows
  and left columns of $D$, and let $D_R$ be the subdiagram formed by
  the intersections of rightist rows and right columns. Notice that
  neither $D_L$ nor $D_R$ have any \di-columns, and the $k$
  bottom rows of $D_L$ have the same length.

  By Lemma~\ref{lem-leftist}, in any partial transversal $F$ of $D$ that
  satisfies (C3), each 1-cell of $F$ is either in $D_L$ or
  in~$D_R$. Thus, $F$ can be decomposed uniquely into
  two transversals $F_L$ and $F_R$, induced by $D_L$ and
  $D_R$, respectively. Conversely, if $F_L$ and $F_R$ are any
  transversals of $D_L$ and $D_R$, then the two fillings give
  rise to a unique partial transversal $F$ of $D$ satisfying (C3).

  Let $F$ be a partial transversal of $D$ that satisfies condition (C3).  Note
  that $F$ satisfies condition (C4) of Observation~\ref{obs-312} if
  and only if $F_L$ avoids 312, and $F$ satisfies (C6) if and only if
  $F_L$ avoids 21 in its bottom $k$ rows. Thus, $F$ satisfies (C4) and
  (C6) if and only if $F_L\in \cFk(D_L,312,21)$. Observe also that $F$
  satisfies (C5) if and only if $F_R$ avoids 12. By Fact~\ref{fac-12},
  this determines $F_R$ uniquely.

  By combining the above remarks, we conclude that a partial transversal $F$ of
  the diagram $D$ avoids 312 if and only if $F_L$ belongs
  to the set $\cFk(D_L,312,21)$ and $F_R$ is the unique
  transversal filling of $D_R$ that avoids~12. By analogous reasoning,
  a partial transversal $F'$ of $D$ avoids 231, if and only if its subfilling
  $F'_L$ induced by $D_L$ belongs to $\cFk(D_L,231,12)$ and
  the subfilling $F'_R$ induced by $D_R$ is the unique transversal
  of $D_R$ avoiding~21.

  The Key Lemma asserts that $\cFk(D_L,312,21)$ and $\cFk(D_L,231,12)$
  have the same cardinality, which implies that the number of
  312-avoiding partial transversals of $D$ is equal to the number of its
  231-avoiding partial transversals.
\end{proof}

The rest of this section is devoted to the proof of the Key Lemma.

Although the proof of the Key Lemma could in principle be presented in
the language of fillings and diagrams, it is more convenient and
intuitive to state the proof in the (equivalent) language of 
matchings. This will allow us to apply previously known results on
pattern-avoiding matchings in our proof.

Let us now introduce the relevant terminology.  A
\emph{matching of order $n$} is a graph $M=(V,E)$ on the vertex set
$V=\{1,2,\dotsc,2n\}$, with the property that every vertex is incident
to exactly one edge. We will assume that the vertices of matchings
are represented as points on a horizontal line, ordered from left to
right in increasing order, and that edges are represented as circular
arcs connecting the two corresponding endpoints and drawn above the
line containing the vertices. If $e$ is an edge connecting vertices
$i$ and $j$, with $i<j$, we say that $i$ is the \emph{left-vertex} and
$j$ is the \emph{right-vertex} of~$e$. Clearly, a matching of order
$n$ has $n$ left-vertices and $n$ right-vertices. Let $\eL(M)$ denote
the set of left-vertices of a matching~$M$.

If $M$ is a matching of order $n$, we define the \emph{reversal} of
$M$, denoted by $\rev M$, to be the matching on the same vertex set as
$M$, such that $\{i,j\}$ is an edge of $\rev M$ if and only if
$\{2n-j+1,2n-i+1\}$ is an edge of~$M$. Intuitively, reversing
corresponds to flipping the matching along a vertical axis.

Let $e=ij$ and $e'=i'j'$ be two edges of a matching $M$, with $i<j$
and $i'<j'$. If $i<i'<j<j'$ we say that \emph{$e$ crosses $e'$ from
  the left} and \emph{$e'$ crosses $e$ from the right}. If
$i<i'<j'<j$, we say that \emph{$e'$ is nested below $e$}. Moreover, if
$k$ is a vertex such that $i<k<j$, we say that $k$ is \emph{nested
  below} the edge $e=ij$, or that $e=ij$ \emph{covers} the vertex~$k$.

A set of $k$ edges of a matching is said to form a \emph{$k$-crossing}
if each two edges in the set cross each other, and it is said to form
a \emph{$k$-nesting} if each two of its edges are nested.

If $M=(V,E)$ is a matching of order $n$ and $M'=(V',E')$ a matching of
order $n'$, we say that \emph{$M$ contains $M'$} if there is an
edge-preserving increasing injection from $V'$ to $V$. In other words,
$M$ contains $M'$ if there is a function $f\colon V'\to V$ such that
for each $u,v\in V'$, if $u<v$ then $f(u)<f(v)$ and if $uv$ is an edge
of $M'$ then $f(u)f(v)$ is an edge of~$M$. If $M$ does not contain
$M'$, we say that \emph{$M$ avoids $M'$}. More generally, if $\cF$ is
a set of matchings, we say that $M$ avoids $\cF$ if $M$ avoids all the
matchings in~$\cF$.

Let $\cM_n$ denote the set of all matchings of order $n$. For a set of
matchings $\cF$ and for a set of integers $X\subseteq[2n]$, define the
following sets of matchings:
\begin{align*}
  \cM_n(X)&=\{M\in\cM_n;\ \eL(M)=X\}\\
  \cM_n(X,\cF)&=\{M\in\cM_n(X);\ M\text{ avoids }\cF\}
\end{align*}

If the set $\cF$ contains a single matching $F$, we will write
$\cM_n(X,F)$ instead of $\cM_n(X,\{F\})$.
                                                   
\begin{figure}                                                
\includegraphics[scale=0.5]{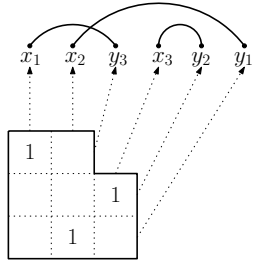}
\caption{The bijection $\mu$ between transversals of Ferrers diagrams and  
matchings. The dotted arrows show the correspondence between rows and columns 
of the diagram and vertices of the matching.}\label{fig-match}
\end{figure}   

De Mier~\cite{adm} has pointed out a one-to-one correspondence between
transversals of (proper) Ferrers diagrams with $n$ rows and $n$ columns
and matchings of order~$n$. This correspondence allows to translate
results on pattern-avoiding transversals of Ferrers diagrams
to equivalent results on pattern-avoiding matchings. We describe the
correspondence here, and state its main properties.

Let $F$ be a transversal of a proper Ferrers diagram $D$.  Let
$n$ be the number of rows (and hence also the number of columns)
of~$D$.  We encode $F$ into a matching $\mu(F)\in \cM_n$ defined as
follows. First, we partition the vertex set $[2n]$ into two disjoint
sets $X(D)=\{x_1<x_2<\dotsb<x_n\}$ and $Y(D)=\{y_1>y_2>\dotsb>y_n\}$,
with the property that $x_j<y_i$ if and only if the $j$-th column of
$D$ intersects the $i$-th row of $D$ (note that the elements of $Y$
are indexed in decreasing order). The diagram $D$ determines $X(D)$
and $Y(D)$ uniquely. Let $\mu(F)$ be the matching whose edge-set is
the set
\[
E=\{x_jy_i;\ F \text{ has a 1-cell in column $j$ and row $i$}\}.
\]                                                  
Figure~\ref{fig-match} shows an example of this correspondence.

We state, without proof, several basic properties of~$\mu$ (see \cite{adm}).

\begin{fact}
  The mapping $\mu$ has the following properties.
  \begin{itemize}
  \item The mapping $\mu$ is a bijection
    between transversals of Ferrers diagrams and matchings, with
    fillings of the same diagram corresponding to matchings with the
    same left-vertices.
    If $F$ is a transversal of a proper Ferrers diagram $D$,
    then $\mu(F)$ is a matching whose left-vertices are precisely the
    vertices from the set~$X(D)$. Conversely, for any matching $M$
    there is a unique proper Ferrers diagram $D$ such that $X(D)$ is
    the set of left-vertices of $M$, and a unique transversal
    $F$ of $D$ satisfying $\mu(F)=M$.
  \item $F$ is a permutation matrix of order $n$ (i.e., a filling of
    an $n\times n$ square diagram) if and only if $\mu(F)$ is a matching
    with $\eL(M)=\{1,2,\dotsc,n\}$.
  \item Assume that $F'$ is a permutation matrix. A filling $F$ avoids the
    pattern $F'$ if and only if the matching $\mu(F)$ avoids the
    matching~$\mu(F')$.
  \item $D$ is a proper Ferrers diagram whose $k$ bottom rows have
    the same length, if and only if $Y(D)$ contains the $k$ numbers
    $\{2n,2n-1,\dotsc,2n-k+1\}$.  In such case, in any matching
    representing a transversal of $D$, all the $k$ rightmost vertices are
    right-vertices.
  \end{itemize}
\end{fact}

\begin{figure}
  \includegraphics[scale=0.5]{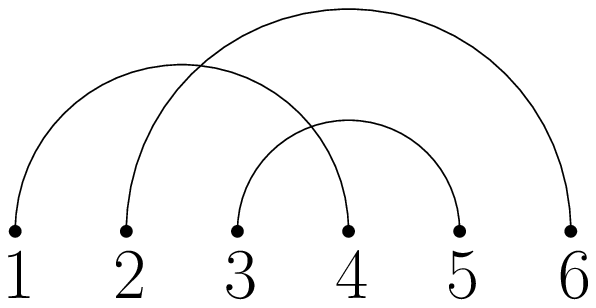}\hfil \includegraphics[scale=0.5]{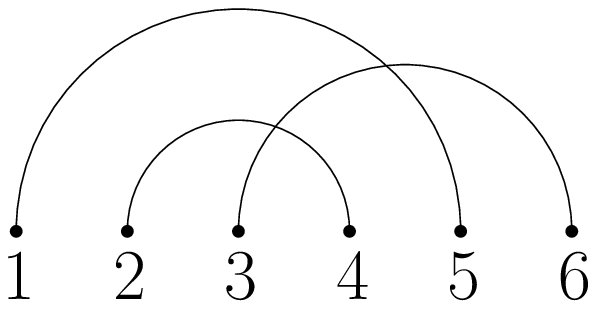}
  \caption{The matching $\mcab$, corresponding to the permutation
    pattern 312 (left), and the matching $\mbca$, corresponding to the
    permutation pattern 231 (right).}\label{fig-m312}
\end{figure}

In the rest of this section, we will say that a matching $M$
\emph{corresponds to} a filling $F$, if $M=\mu(F)$. We will also say
that $M$ corresponds to a permutation $p$ if it corresponds to the
permutation matrix of~$p$.  Specifically, we let $\mcab$ be the
matching corresponding to the permutation 312, and let $\mbca$ be the
matching corresponding to the permutation 231 (see
Fig.~\ref{fig-m312}).

Let $D$ be a proper Ferrers diagram with $n$ rows and $n$ columns, whose
bottom $k$ rows have the same length. To prove the Key Lemma, we need
a bijection between the sets of fillings $\cFk(D,312,21)$ and
$\cFk(D,231,12)$. Let $\cMk(D,312,21)$ be the set of matchings that
correspond to the fillings from the set $\cFk(D,312,21)$, and
similarly let $\cMk(D,231,12)$ be the set of matchings corresponding
to the fillings from $\cFk(D,231,12)$.

By definition, a matching $M$ belongs to $\cMk(D,312,21)$ if and only
if $\eL(M)=X(D)$, $M$ avoids $\mcab$, and the $k$ edges incident to
the rightmost $k$ vertices of $M$ form a $k$-nesting. (Notice that all
the rightmost $k$ vertices of $M$ are right-vertices, since the bottom
$k$ rows of $D$ are assumed to have the same length.) 
Similarly, a matching $M$ belongs to $\cMk(D,231,12)$
if and only if $\eL(M)=X(D)$, $M$ avoids $\mbca$, and the edges
incident to the rightmost $k$ vertices form a $k$-crossing.

Let $M$ be a matching. A sequence of edges $(e_1,e_2,\dotsc,e_p)$ is
called \emph{a chain of order $p$ from $e_1$ to $e_p$}, if for each
$i< p$, the edge $e_i$ crosses the edge $e_{i+1}$ from the left. A
chain is \emph{proper} if each of its edges only crosses its neighbors
in the chain. It is not difficult to see that every chain from $e_1$
to $e_p$ contains, as a subsequence, a proper chain from $e_1$
to~$e_p$.

\begin{figure}[ht]
\includegraphics[width=0.6\textwidth]{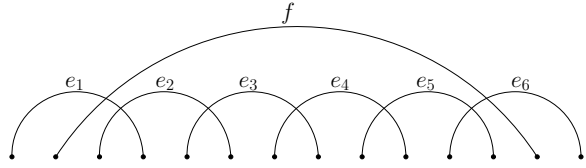}
\caption{The cyclic chain of order 7.}\label{fig-cyclic}
\end{figure}   

A \emph{cyclic chain of order $p+1$} is a $(p+1)$-tuple of edges
$(f,e_1,\dotsc, e_p)$, with the following properties.
\begin{itemize}
\item The sequence $(e_1,\dotsc,e_p)$ is a proper chain.
\item The edge $f$ crosses $e_1$ from the right and $e_p$ from the
  left. Furthermore, for each $i\in\{2,3,\dotsc,p-1\}$, the edge $e_i$
  is nested below~$f$.
\end{itemize}
Figure~\ref{fig-cyclic} shows an example of a cyclic chain of order~7.
The matching of order $p+1$ whose edges form a cyclic chain will be
denoted by~$C_{p+1}$. The smallest cyclic chain is $C_3$, whose three
edges form a 3-crossing. Let $\cC$ denote the infinite set
$\{C_q\colon q\ge 3\}$.

As shown in~\cite{jel}, there is a bijection $\psi$ which maps the set
of $\mcab$-avoiding matchings to the set of $\cC$-avoiding matchings,
with the additional property that each $\mcab$-avoiding matching $M$
is mapped to a $\cC$-avoiding matching $\psi(M)$ with the same order
and the same set of left-vertices. Since the reversal of a
$\mcab$-avoiding matching is an $\mbca$-avoiding matching, while the
reversal of a $\cC$-avoiding matching is again $\cC$-avoiding, it is
easy to see that the mapping $M\mapsto \rev{\psi(\rev M)}$ is a
bijection that maps an $\mbca$-avoiding matching $M$ to a
$\cC$-avoiding matching with the same set of left-vertices.

We will use the bijection $\psi$ as a building block of our bijection
between the sets $\cMk(D,312,21)$ and $\cMk(D,231,12)$. However, before we do
so, we need to describe the bijection~$\psi$, which requires more
terminology.

Let $M$ be a matching on the vertex set $[2n]$. For an integer
$r\in[2n]$, we let $M[r]$ denote the subgraph of $M$ induced by the
leftmost $r$ vertices of~$M$. We will call $M[r]$ \emph{the $r$-th
  prefix of $M$}. The graph $M[r]$ is a union of disjoint edges and
isolated vertices. The isolated vertices of $M[r]$ will be called
\emph{the stubs} of $M[r]$.

If $x$ and $x'$ are two stubs of $M[r]$, with $x<x'$, we say that $x$
and $x'$ are \emph{equivalent in $M[r]$}, if $M[r]$ contains a chain
$(e_1,\dotsc,e_p)$ (possibly containing a single edge) such that $x$
is nested below $e_1$ and $x'$ is nested below~$e_p$. We will also
assume that each stub is equivalent to itself. As shown in~\cite{jel},
this relation is indeed an equivalence relation on the set of
stubs. The blocks of this equivalence relation
will be simply called \emph{the blocks of $M[r]$}. It is easy to see
that if $x$ and $x'$ are stubs belonging to the same block, and $x''$
is a stub satisfying $x<x''<x'$, then $x''$ belongs to the same block
as $x$ and~$x'$. Figure~\ref{fig-stubs} shows an example.
                                            
\begin{figure}[ht]
\includegraphics[width=0.8\textwidth]{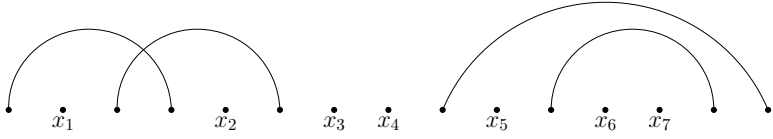}
\caption{A prefix of a matching with seven stubs, forming four equivalence classes 
$\{x_1,x_2\}$, $\{x_3\}$, $\{x_4\}$, and $\{x_5,x_6,x_7\}$.}\label{fig-stubs}
\end{figure}

For a matching $M\in\cM_n$, the sequence of prefixes $M[1],
M[2],\dotsc, M[2n]=M$ will be called \emph{the generating sequence of
  $M$}. We will interpret this sequence as a sequence of steps of an
algorithm that generates the matching $M$ by adding vertices
one-by-one, from left to right, starting with the graph $M[1]$ that
consists of a single isolated vertex.

Each prefix in the generating sequence defines an equivalence on the
set of its stubs. To describe the bijection $\psi$, we first need to
point how the blocks of these equivalences change when we pass from
one prefix in the sequence to the next one.

Clearly, $M[1]$ consists of a single stub, so its equivalence has a
single block $\{1\}$. Let us now show how the equivalence defined by
$M[r]$ differs from the equivalence defined by $M[r-1]$. If a vertex
$r>1$ is a left-vertex of $M$, then the graph $M[r]$ is obtained from
$M[r-1]$ by adding a new stub~$r$. In such case, we say that $M[r]$ is
obtained from $M[r-1]$ by \emph{an L-step}. It is obvious that each
block of $M[r-1]$ is also a block of $M[r]$, and apart from that
$M[r]$ also has the singleton block~$\{r\}$.

Assume now that $r>1$ is a right-vertex of $M$. In this situation, we
say that $M[r]$ is obtained from $M[r-1]$ by \emph{an
  R-step}. Clearly, $M[r]$ is obtained from $M[r-1]$ by adding the
vertex $r$ and connecting it by an edge to a stub $s$ of~$M[r-1]$.  We
say that the stub $s$ is \emph{selected in step $r$}. In such case,
$s$ is no longer a stub in $M[r]$. Let $B_1, B_2,\dotsc, B_b$ be the
blocks of $M[r-1]$ ordered left to right, and assume that $s$ belongs
to a block~$B_j$. Then $B_1, B_2,\dotsc, B_{j-1}$ are also blocks in
$M[r]$. The set $(B_j\setminus\{s\})\cup B_{j+1}\cup\dotsb\cup B_b$ is
either empty or forms a block of~$M[r]$. Notice that the sizes of the
blocks of $M[r]$ only depend on the value of $j$ and on the sizes of
the blocks of~$M[r-1]$. We define two special
types of R-steps: a \emph{maximalist} R-step is an R-step in which the
selected stub is the rightmost stub of its block (i.e., $s=\max B_j$),
while a \emph{minimalist} R-step is an R-step in which the selected
stub is the leftmost stub in its block.

To connect our terminology with the results from~\cite{jel}, we need a
simple lemma.
\begin{lemma}\label{lem-equiv} 
  Let $M\in \cM_n$ be an $\mcab$-avoiding matching, let $r\in [2n]$ be
  an integer. Let $s$ and $s'$ be two distinct stubs of $M[r]$. The
  two stubs $s$ and $s'$ belong to the same block, if and only if
  $M[r]$ has an edge $e$ that covers both $s$ and~$s'$.
\end{lemma}
\begin{proof} 
  By definition, if two stubs are covered by a single edge of $M[r]$,
  they are equivalent and hence belong to the same block. To prove the
  converse, assume that $s< s'$ are stubs of $M[r]$ that belong to the
  same block. Let $C=(e_1,\dotsc,e_p)$ be a chain in $M[r]$, such that
  $e_1$ covers $s$ and $e_p$ covers~$s'$. Choose $C$ to be as short as
  possible. If $C$ consists of a single edge, then $s$ and $s'$ are
  both covered by this edge and we are done. For contradiction, assume
  that $C$ has at least two edges. The edge $e_2$ does not cover $s$,
  because if it did, the chain $(e_2,\dotsc, e_p)$ would contradict
  the minimality of~$C$. Let $f$ be the edge of $M$ incident to the
  vertex~$s$. Necessarily, the right endpoint of $f$ is greater than
  $r$, otherwise $s$ would not be a stub in~$M[r]$. In particular, in
  the matching $M$, $f$ intersects $e_1$ from the right, and $e_2$ is
  nested below~$f$. Thus, the three edges $e_1, e_2$ and $f$ form in
  $M$ a copy of $\mcab$, contradicting the assumption that $M$ is
  $\mcab$-avoiding.
\end{proof}

Combining Lemma~\ref{lem-equiv} with \cite[Lemma 3]{jel}, we get the
following result that gives characterizations of $\mcab$-avoiding and
$\cC$-avoiding matchings.

\begin{fact}\label{fac-jel} 
  A matching $M\in\cM_n$ avoids the pattern $\mcab$ if and only if,
  for every right-vertex $r>1$ of $M$, $M[r]$ is obtained from
  $M[r-1]$ by a minimalist R-step. A matching $M\in\cM_n$ avoids the
  set of patterns~$\cC$ if and only if, for every right-vertex $r>1$
  of $M$, $M[r]$ is obtained from $M[r-1]$ by a maximalist R-step.
\end{fact}

We are now ready to state the following key result from~\cite{jel},
which describes the properties of the bijection~$\psi$.

\begin{fact}\label{fac-psi} 
  There is a bijection $\psi$ between $\mcab$-avoiding and
  $\cC$-avoiding matchings. If $M$ is an $\mcab$-avoiding matching of
  order $n$, and $N=\psi(M)$ its corresponding $\cC$-avoiding
  matching, then the following holds.
  \begin{itemize}
  \item $M$ and $N$ have the same set of left-vertices (and hence the
    same size).
  \item For any vertex $r\in[2n]$, the prefix $M[r]$ has the
    same number of blocks as the prefix $N[r]$. Moreover, if
    $B_1,\dotsc, B_b$ are the blocks of $M[r]$ in left-to-right order,
    and $B'_1,\dotsc, B'_b$ are the blocks $N[r]$ in left-to-right
    order, then $|B_i|=|B'_i|$ for each~$i\le b$.
  \item Assume that $r+1$ is a right-vertex of $M$ (and hence also of
    $N$), and that $B_1,\dotsc, B_b$ and $B'_1,\dotsc,B'_b$ are blocks
    of $M[r]$ and $N[r]$, as above. If $M[r+1]$ is obtained from
    $M[r]$ by selecting a stub $s$ from a block $B_j$, then $N[r+1]$
    is obtained from $N[r]$ by selecting a stub $s'$ from the
    corresponding block~$B'_j$. In view of Fact~\ref{fac-jel}, we must
    then have $s=\min B_j$ and $s'=\max B'_j$.
  \end{itemize}
\end{fact}

The properties of $\psi$ listed above in fact determine $\psi$
uniquely.

Finally, we are ready to present the bijection between
$\cMk(D,312,21)$ and $\cMk(D,231,12)$. Recall that the matchings from
these two sets have the same set of
left-vertices $X(D)$ and the same set of right-vertices
$Y(D)=[2n]\setminus X(D)$. Let us write $X(D)=\{x_1<x_2<\dotsc<x_n\}$
and $Y(D)=\{y_1>y_2>\dotsc >y_n\}$. Recall also that by assumption,
the rightmost $k$ right-vertices $y_1,\dotsc,y_k$ are to the right of
any left-vertex.

The bijection we present is a composition of several steps, with the
correctness of each step proved separately. An example is shown in 
Figure~\ref{fig-biject}.

\noindent\textbf{Step 1: apply $\psi$.} Use the bijection~$\psi$ to
map the set $\cMk(D,312,21)$ bijectively to the set $S_1=\{\psi(M);\,
M\in\cMk(D,312,21)\}$. As shown in Lemma~\ref{lem-step1} below, $S_1$
is precisely the set of all the matchings $N$ satisfying the following
properties:
\begin{itemize}
\item[(P1)] $N$ avoids $\cC$.
\item[(P2)] $\eL(N)=X(D)$.
\item[(P3)] In the prefix $N[2n-k]$ of $N$, each block has a single stub.
\item[(P4)] The edges of $N$ incident to $y_1,\dotsc,y_k$ form a
  $k$-nesting.
\end{itemize}

\noindent\textbf{Step 2: add edge.} For a matching $M\in S_1$, let
$M^+$ be the matching obtained from $M$ by adding two new vertices
$x_\new$ and $y_\new$ and a new edge $e_\new=x_\new y_\new$, such that
the edge $e_\new$ covers precisely the vertices $y_1,\dotsc, y_k$. We
relabel the $2n+2$ vertices of $M^+$, without altering their
left-to-right order, so that their labels correspond to the integers
$1,\dotsc, 2n+2$ in their usual order. With this labeling, we have
$x_\new=2n-k+1$ and $y_\new=2n+2$.

Let $S_2$ be the set $\{M^+;\; M\in S_1\}$. Clearly, all the matchings
in $S_2$ share the same set of left-vertices and the same set of right-vertices.
We call these sets $X^+$ and $Y^+$, respectively. It is
also clear that the mapping $M\mapsto M^+$ is a bijection between
$S_1$ and $S_2$. In Lemma~\ref{lem-step2}, we will show that $S_2$ is
precisely the set of all the matchings $N$ that satisfy the following
conditions:
\begin{itemize}
\item[(R1)] $N$ avoids $\cC$.
\item[(R2)] $\eL(N)=X^+$.
\item[(R3)] $N$ contains the edge $e_\new=\{2n-k+1,2n+2\}$.
\end{itemize}

\noindent\textbf{Step 3: reverse.} Recall that $\rev M$ denotes the
reversal of a matching~$M$. Let $S_3$ be the set $\{\rev M;\; M\in
S_2\}$. All the matchings in $S_3$ have the same set of left-vertices
and right-vertices, denoted by $\rev{X^+}$ and $\rev{Y^+}$,
respectively. From the previously stated properties of $S_2$ it
follows that $S_3$ contains precisely the matchings $N$ satisfying
these conditions:
\begin{itemize}
\item[($\rev{\text{R1}}$)] $N$ avoids $\cC$.
\item[($\rev{\text{R2}}$)] $\eL(N)=\rev{X^+}$.
\item[($\rev{\text{R3}}$)] $N$ contains the edge $\{1,k+2\}$.
\end{itemize}

\noindent\textbf{Step 4: apply $\psi^{-1}$.} Let $S_4$ be the set
$\{\psi^{-1}(M);\; M\in S_3\}$. In Lemma~\ref{lem-step4}, we show that
$S_4$ contains precisely the matchings $N$ satisfying these three
conditions:
\begin{itemize}
\item[(S1)] $N$ avoids $\mcab$.
\item[(S2)] $\eL(N)=\rev{X^+}$.
\item[(S3)] $N$ contains the edge $\{1,k+2\}$.
\end{itemize}

\noindent\textbf{Step 5: remove edge.} For a matching $M\in S_4$, let
$M^-$ denote the matching obtained from $M$ by removing the edge
$\{1,k+2\}$ together with its endpoints. Relabel the vertices of $M^-$
by integers $1,2,\dotsc,2n$, in their usual order. Let $S_5$ be the
set $\{M^-;\; M\in S_4\}$. All the matchings in $S_5$ have the same
set of left-vertices, denoted by~$\rev X$. We show in
Lemma~\ref{lem-step5} that $S_5$ contains precisely the following
matchings $N$:
\begin{itemize}
\item[(S1$^-$)] $N$ avoids $\mcab$.
\item[(S2$^-$)] $\eL(N)=\rev{X}$.
\item[(S3$^-$)] The edges incident to the leftmost $k$ vertices of $N$
  form a $k$-crossing.
\end{itemize}

\noindent\textbf{Step 6: reverse back.} The properties of $S_5$ stated
above imply that the matchings in $S_5$ are exactly the reversals of
the matchings in $\cMk(D,231,12)$. Thus, applying reversal to the
elements of $S_5$ we complete the bijection from $\cMk(D,312,21)$ to
$\cMk(D,231,12)$.

\begin{figure}
\includegraphics[height=0.9\textheight]{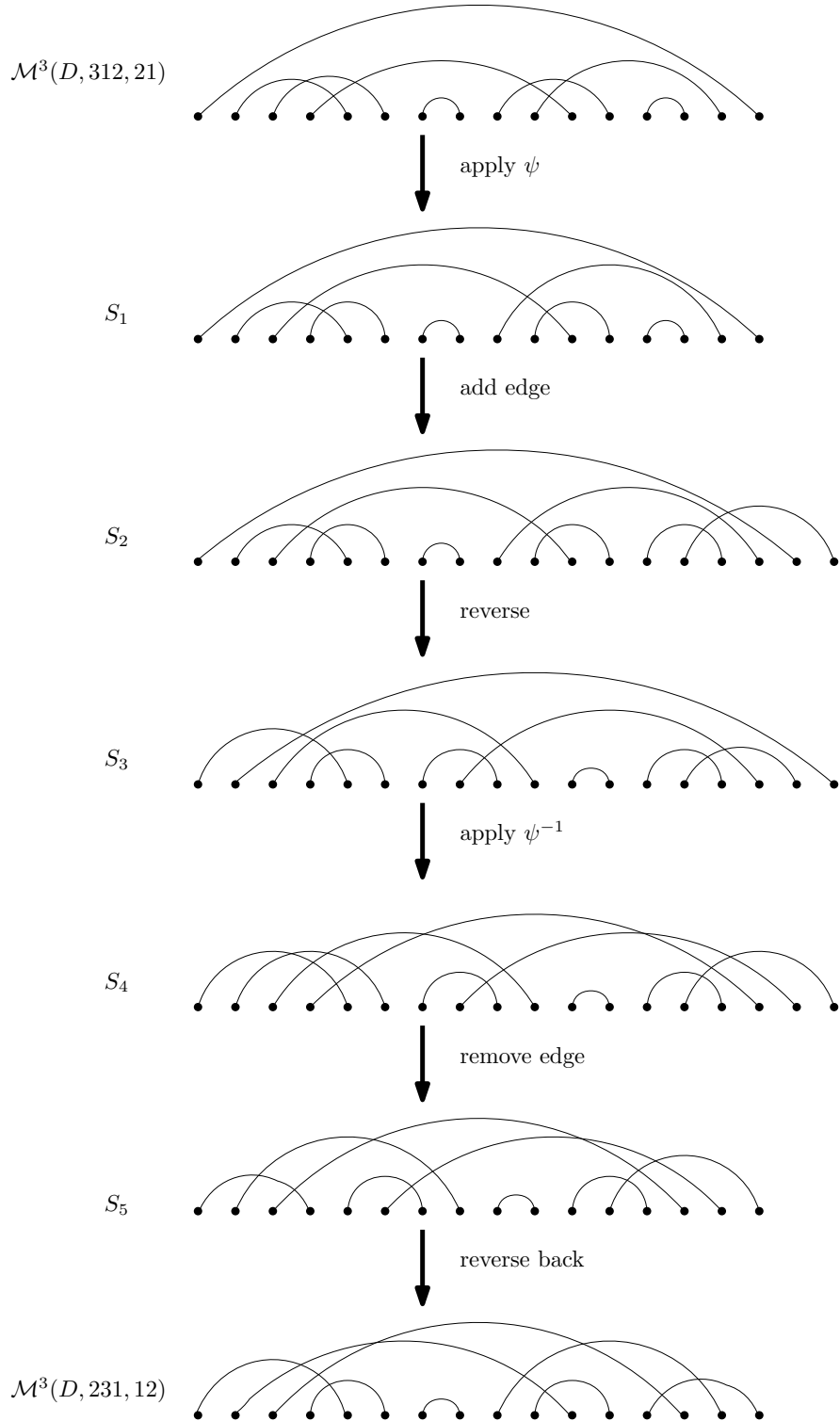}
\caption{The six steps of a bijection from $\cMk(D,312,21)$ to $\cMk(D,231,12)$ (with 
$k=3$).}\label{fig-biject}
\end{figure}

Next, we will prove the correctness of the individual steps. The proofs are 
mostly routine.

\begin{lemma}\label{lem-step1}
  The set $S_1$ contains precisely those matchings that satisfy the four
  properties (P1)--(P4).
\end{lemma}

\begin{proof}
  Let $N$ be a matching from the set $S_1$. Let $M\in\cMk(D,312,21)$ be
  the preimage of $N$ under~$\psi$. The properties of $\psi$ stated in
  Fact~\ref{fac-psi} directly show that $N$ satisfies (P1) and (P2).
  
  We now show that $N$ satisfies (P3). Fact~\ref{fac-psi} shows that
  $N$ satisfies (P3) if and only if $M$ satisfies (P3). It is thus
  enough to prove (P3) for~$M$.

  In the matching $M$, the $k$ edges incident to $y_1,\dotsc, y_k$
  form a $k$-nesting, by the definition of $\cMk(D,312,21)$. Assume
  that $M$ does not satisfy (P3), i.e., in the prefix
  $M[2n-k]$, there are two stubs $s<s'$ belonging to the same
  block. By Lemma~\ref{lem-equiv}, this means that $s$ and $s'$ are
  both covered by a single edge $e\in M[2n-k]$. Let $f$ and $f'$ be
  the edges of $M$ incident to $s$ and $s'$, respectively. The right
  endpoints of $f$ and $f'$ must belong to $\{y_1,\dotsc,y_k\}$, which
  means that $f$ and $f'$ are nested. This means that $e$, $f$, and
  $f'$ form a copy of $\mcab$ in $M$, which is impossible.

  Next, we show that $N$ satisfies (P4), i.e., that the edges of $N$
  incident to $y_1,\dotsc,y_k$ form a $k$-nesting. This is equivalent
  to saying that for every $i\in [k]$, the prefix
  $N[2n-i+1]$ is obtained from $N[2n-i]$ by adding the right-vertex
  $y_i$ and connecting it to the rightmost stub of $N[2n-i]$. We know
  that the edges of $M$ incident to $\{y_1,\dotsc,y_k\}$ form a
  $k$-nesting. Hence, for each $i\in [k]$, $M[2n-i+1]$ is obtained from
  $M[2n-i]$ by an R-step in which the rightmost stub in $M[2n-i]$ is
  selected. By the properties of $\psi$, we also select the rightmost
  stub whenever we create $N[2n-i+1]$ from $N[2n-i]$. This shows that
  $N$ has property~(P4).

  We now show that every matching satisfying the four
  properties (P1)--(P4) belongs to $S_1$. Let $N$ be a matching
  satisfying (P1)--(P4). Since $N$ is $\cC$-avoiding, we may define
  $M=\psi^{-1}(N)$. To show that $N$ belongs to $S_1$, we need to
  prove that $M$ belongs to $\cMk(D,312,21)$. The properties of $\psi$
  guarantee that $M$ is $\mcab$-avoiding and that $\eL(M)=X(D)$. It
  remains to show that rightmost $k$ vertices of $M$ are incident to a
  $k$-nesting.

  Since $N$ satisfies (P3), so does $M$. Moreover, since $N$ satisfies
  (P4), we know that for each $i\le k$, the prefix
  $N[2n-i+1]$ is obtained from $N[2n-i]$ by adding the vertex $y_i$
  and connecting it to the rightmost stub of $N[2n-i]$. From this
  description, we easily notice that for each $i\le k$, each block in
  the matching $N[2n-i]$ is a singleton. By the properties of $\psi$,
  each block of $M[2n-i]$ is also a singleton, and $M[2n-i+1]$ is
  created form $M[2n-i]$ by connecting $y_i$ to the rightmost stub of
  $M[2n-i]$. This shows that in $M$, the vertices $y_1,\dotsc, y_k$
  are indeed incident to a $k$-nesting, and hence $M\in
  \cMk(D,312,21)$.
\end{proof}

\begin{lemma}\label{lem-step2} 
  A matching $N$ belongs to $S_2$ if and only if it satisfies
  (R1)--(R3).
\end{lemma}
\begin{proof}
  Suppose $N\in S_2$. Let $M$ be a matching from $S_1$ such that
  $N=M^+$. By construction, $N$ satisfies (R2) and (R3). We need to
  show that $N$ also satisfies (R1), i.e., that it is
  $\cC$-avoiding. For contradiction, assume that $N$ contains a copy
  $C$ of a cyclic chain formed by $p+1$ edges
  $(f,e_1,\dotsc,e_p)$. Recall that in a cyclic chain, the sequence
  $(e_1,\dotsc,e_p)$ is a proper chain, and $f$ crosses $e_1$ from the
  right and $e_p$ from the left, while the edges $e_2,\dotsc,e_{p-1}$
  are nested below~$f$.

  Since $M$ is $\cC$-avoiding, $C$ must contain the new edge
  $\{2n-k+1, 2n+2\}$. Necessarily, the new edge is the edge $e_p$,
  which is incident to the rightmost vertex of $C$. In the matching
  $M$, the edges incident to $y_1,\dotsc,y_k$ form a $k$-nesting. In
  particular, the edges $f$ and $e_{p-1}$ are nested, and hence $C$
  has at least four edges. Let $s$ and $s'$ be the left endpoints of
  the edges $f$ and $e_{p-1}$, respectively. Consider now the prefix
  $M[2n-k]$. This prefix contains the nonempty
  proper chain $e_1,\dotsc,e_{p-2}$ (possibly consisting of a single
  edge), and the two vertices $s$ and $s'$ are stubs of
  $M[2n-k]$. Since $s$ is covered by $e_1$ and $s'$ is covered by
  $e_{p-2}$, the two stubs belong to the same block of $M[2n-k]$,
  which is impossible, since the matching $M\in S_1$ must satisfy
  (P3).

  We conclude that every matching from $S_2$ satisfies (R1)--(R3).

  To prove the converse, assume that $N$ is a matching satisfying
  (R1)--(R3). By (R3), there is a matching $M$ such that $M^+=N$. We
  need to show that $M$ belongs to $S_1$, i.e., that it satisfies
  (P1)--(P4). It is clear that $M$ satisfies (P1) and (P2). If $M$
  fails (P4), then $N$ must contain a 3-crossing, which is impossible,
  since a 3-crossing is a special case of a cyclic chain. If $M$
  satisfies (P4) but fails (P3), then $N$ contains a cyclic chain of
  length at least four, which is also impossible. Thus $M$ belongs to
  $S_1$, and hence $N$ belongs to $S_2$, as claimed.
\end{proof}

\begin{lemma}\label{lem-step4} 
  A matching $N$ belongs to $S_4$ if and only if it satisfies
  (S1)--(S3).
\end{lemma}
\begin{proof}
  Choose $N\in S_4$, and set $M=\psi(N)$. By definition, $M$ belongs
  to $S_3$, so it satisfies
  ($\rev{\text{R1}}$)--($\rev{\text{R3}}$). It follows directly that
  $N$ satisfies (S1) and (S2).

  We know that $M$ satisfies ($\rev{\text{R3}}$). Consider the R-step from
  $M[k+1]$ to $M[k+2]$. Since $M[k+1]$ consists of $k+1$ stubs, all
  its blocks are singletons. Since $M$ has the edge $\{1,k+2\}$, the
  prefix $M[k+2]$ has been constructed from $M[k+1]$ by
  selecting the leftmost stub of $M[k+1]$ and connecting it to the
  vertex~$k+2$. By the properties of $\psi$, this means that $N[k+2]$
  was obtained from $N[k+1]$ in the same way, and in particular, $N$
  contains the edge $\{1,k+2\}$. We conclude that $N$ satisfies
  (S1)--(S3).

  The same argument shows that every matching satisfying (S1)--(S3)
  belongs to~$S_4$.
\end{proof}

\begin{lemma}\label{lem-step5}
  A matching $N$ belongs to $S_5$ if and only if it satisfies
  (S1$^-$)--(S3$^-$).
\end{lemma}
\begin{proof} 
  Choose $N\in S_5$, and let $M$ be the matching from $S_4$ such that
  $N=M^-$. Clearly, $N$ satisfies (S1$^-$) and (S2$^-$). Assume for
  contradiction that $N$ fails (S3$^-$). In such case, $N$ has two
  nested edges $e_1$ and $e_2$, whose left endpoints are among the
  leftmost $k$ vertices of~$N$. Since the leftmost $k$ vertices of $N$
  are all left-vertices, we know that in the matching $M$, the two
  edges $e_1$ and $e_2$ are both crossed from the left by the edge
  $\{1,k+2\}$, forming the pattern $\mcab$ forbidden by~(S1).

  It is easy to see that any matching satisfying (S1$^-$)--(S3$^-$) belongs to~$S_5$.
\end{proof}

This completes the proof of Lemma~\ref{lem-key312}, and hence also of
Theorem~\ref{thm-312}.

\section{The $k$-Wilf-equivalence of patterns of length $k+2$}
\label{sec-bax}

We will now consider the structure of pattern-avoiding partial
permutations in which the number of holes is close to the length of
the forbidden pattern.
                                               
Let us begin by an easy observation.

\begin{observation}\label{obs-l}
Let $p$ be a pattern of length $\ell$. Obviously any partial permutation with 
at least $\ell$ holes contains~$p$. Almost as obviously, a 
partial permutation with $\ell-1$ holes and of length at least 
$\ell$, contains $p$ as well. In particular, 
$s_n^k(p) = 0$ for every $k\ge \ell-1$ and $n\ge\ell$, and all patterns of length $\ell$ are 
$k$-Wilf-equivalent.                                  
\end{observation}  

In the rest of this section, we will deal with $k$-Wilf-equivalence of
patterns of length $\ell=k+2$. As we will see, an important part 
will be played by Baxter permutations, which we now define.

\begin{definition} 
  A permutation $p\in \sym_\ell$ is called a \emph{Baxter permutation},
  if there is no four-tuple of indices $a<b<c<d\in [\ell]$ such that
  \begin{itemize}
  \item $c=b+1$, and
  \item the subpermutation $p_a,p_b,p_c,p_d$ is order-isomorphic to $2413$ or to $3142$.
  \end{itemize}
  In the terminology of Babson and Steingr\'{\i}msson~\cite{BS},
  Baxter permutations are exactly the permutations avoiding simultaneously the 
  two patterns 2-41-3 and 3-14-2.
\end{definition}

Baxter permutations were originally introduced by G. Baxter~\cite{bax} in
1964, in the study of common fixed
points of commuting continuous functions~\cite{bax,boy}. Later, it has
been discovered that Baxter permutations are also closely related to
other combinatorial structures, such as plane bipolar
orientations~\cite{bon}, noncrossing triples of lattice
paths~\cite{fel}, and standard Young tableaux~\cite{dul}. An explicit
formula for the number of Baxter permutations has been found by Chung
et al.~\cite{chu}, with several later refinements~\cite{mal,vie,dul2}.

To deal with $k$-Wilf equivalence of patterns of length $k+2$, we first need 
to introduce more notation. Let $\pi\in \sym_n^H$ be a partial permutation, with 
$|H|=k$. Let $h_1<h_2<\dotsb<h_k$ be the elements of~$H$.  Let $I$ denote the 
set $[n]\setminus H$, i.e., $I$ is the set of indices of the non-holes 
of $\pi$. We may decompose the set $I$ into $k+1$ (possibly empty) intervals 
$I_1, I_2,\dotsc, I_{k+1}$, by defining $I_1=\{i\in I;\; i<h_1\}$, 
$I_{k+1}=\{i\in I;\; i>h_k\}$, and for each $a\in\{2,\dotsc,k\}$, $I_a=\{i\in 
I;\; h_{a-1}<i<h_a\}$.

\begin{lemma}\label{lem-bax1} 
  Let $\ell$ and $n$ be integers, let $k=\ell-2$. Let $p=p_1\dotsb
  p_\ell$ be a permutation and let $\pi=\pi_1\dotsb\pi_n$ be a
  partial permutation with $k$ holes. Assume that $H, I,
  I_1,\dotsc, I_{k+1}$ are as above. The partial permutation $\pi$ avoids the
  pattern $p$ if and only if for each two distinct indices $i\in I_a$
  and $j\in I_b$ such that $i<j$, the relative order of $\pi_i$ and
  $\pi_j$ is different from the relative order of $p_a$ and~$p_{b+1}$ 
  (i.e., $\pi_i<\pi_j\iff    p_a>p_{b+1}$).

  Consequently, for each such $p$, $n$ and $H$, we have $s_n^H(p)\le
  1$.
\end{lemma}
\begin{proof} 
  Assume that $i<j$ are distinct indices from the set $I$, with $i\in
  I_a$ and $j\in I_b$. Necessarily, $a\le b$. Note that in $\pi$ there
  are $a-1$ holes to the left of $\pi_i$, there are $b-a$
  holes between $\pi_i$ and $\pi_j$, and there are $k-b+1$
  holes to the right of~$\pi_j$.

  Assume that for some $i\in I_a$ and $j\in I_b$, with $i<j$, the
  symbols $\pi_i$ and $\pi_j$ have the same relative order as $p_a$
  and $p_{b+1}$. Then $\pi$ contains an occurrence of $p$, in which
  $\pi_i$ corresponds to $p_a$, $\pi_j$ corresponds to $p_{b+1}$, and
  the $k$ holes correspond to the remaining $k$ symbols
  of~$p$.

  Conversely, assume that $\pi$ contains an occurrence of $p$. This
  means that there is an $\ell$-tuple $P$ of indices, such that the
  subsequence $(\pi_h;\; h\in P)$ is a copy of~$p$. It is not hard to
  see that in such case we may always find a copy of $p$ that uses all
  the $k$ holes of~$\pi$. In other words, we may assume that $H$ is a
  subset of~$P$. Let $i$ and $j$ be the two indices of $P$ not
  belonging to~$H$, with $i<j$. Fix $a$ and $b$ such that $i\in I_a$
  and~$j\in I_b$. In the $\ell$-tuple $(\pi_h;\; h\in P)$, the element
  $\pi_i$ is the $a$-th element, since it has $a-1$ holes to
  the left of it, while $j$ is $(b+1)$-th element, since it has $b-1$
  holes and the symbol $\pi_i$ to the left of it. Since
  $(\pi_h;\; h\in P)$ is assumed to be a copy of~$p$, we conclude that
  $\pi_i$ and $\pi_j$ have the same relative order as $p_a$
  and~$p_{b+1}$.

  This shows that $\pi$ avoids $p$ if and only if for each two
  distinct indices $i<j$ with $i\in I_a$ and $j\in I_b$, the relative
  order of $\pi_i$ and $\pi_j$ differs from the relative order of
  $p_a$ and~$p_{b+1}$.

  For a fixed $p\in \sym_\ell$, for each $n$ and for each set $H\subseteq
  [n]$ of size $k$, if $\pi$ is a partial permutation from $\sym_n^H(p)$, the
  relative order of every two non-holes of $\pi$ is uniquely
  determined by the relative order of the symbols of~$p$. In
  particular, $\pi$ is uniquely determined by $p$, $n$ and $H$,
  implying that~$s_n^H(p)\le 1$.
\end{proof}

Motivated by Lemma~\ref{lem-bax1}, we introduce the following
notation. Let $p\in \sym_\ell$ be a pattern, let $k=\ell-2$, let $n$ be
an integer, and let $H\subseteq[n]$ be a $k$-element set of
integers. The \emph{order graph $G_n^H(p)$} is a directed graph on the
vertex set $I=[n]\setminus H$, whose edge-set is defined by the
following condition: for every $i<j$, such that $i\in I_a$ and $j\in
I_b$, the graph $G_n^H(p)$ has an edge from $i$ to $j$ if
$p_a>p_{b+1}$, and it has an edge from $j$ to $i$ if $p_a<p_{b+1}$.

Note that $G_n^H(p)$ is a tournament, i.e., for each pair of distinct
vertices $i$ and $j$, the graph $G_n^H(p)$ has an edge from $i$
to $j$ or an edge from $j$ to $i$, but not both.

Let $\pi=\pi_1\dotsb\pi_n$ be a partial permutation from the set $\sym_n^H$.
Using the notion of order graphs, Lemma~\ref{lem-bax1} can be restated
in the following equivalent way: $\pi$ avoids $p$ if and only if, for
each two distinct vertices $i,j$ of $G_n^H(p)$, if the graph
$G_n^H(p)$ has a directed edge from $i$ to $j$ then
$\pi_i<\pi_j$. Notice that in this statement, we no longer need to
assume that $i<j$.

\begin{lemma}\label{lem-graph}
  Let $p\in \sym_\ell$ be a pattern, let $k=\ell-2$, let $n$ be an
  integer, and let $H\subseteq[n]$ be a $k$-element set of
  integers. The following statements are equivalent:
  \begin{enumerate}
  \item $s_n^H(p)=1$.
  \item $G_n^H(p)$ has no directed cycle.
  \item $G_n^H(p)$ has no directed cycle of length 3.
  \end{enumerate}
\end{lemma}
\begin{proof}
  Since $G_n^H(p)$ is a tournament, the statements 2 and 3 are easily
  seen to be equivalent.
  
  Let us now show that (1) implies (2). Assume that $s_n^H(p)=1$, and
  let $\pi=\pi_1\dotsb \pi_n$ be the partial permutation from $\sym_n^H(p)$. As
  we have pointed out before, if $G_n^H(p)$ has an edge from $i$ to
  $j$, then $\pi_i<\pi_j$. This clearly shows that $G_n^H(p)$ may have
  no directed cycle.

  Conversely, if $G_n^H(p)$ has no directed cycle, we may
  topologically order its vertices, i.e., we can assign to every
  vertex $i$ a value $\pi_i$ in such a way that if the graph has an
  edge from $i$ to $j$, then $\pi_i<\pi_j$. The values $\pi_i$ then
  define a $p$-avoiding partial permutation $\pi\in \sym_n^H(p)$, showing that
  $s_n^H(p)=1$.
\end{proof}

We are now ready to demonstrate the significance of Baxter permutations. 
Note that for any pattern of length $\ell=k+2$, and for any $n$ from the
set $\{k,k+1,k+2\}$, we always have $s_n^k(p)=\binom{n}{k}$. 
Thus, for these small values of $n$, all patterns have the same behavior.
However, for all larger values of $n$, the Baxter patterns are separated from the rest, as
the next proposition shows.

\begin{proposition}\label{pro-baxter} 
  Let $p$ be a permutation pattern of size~$\ell$, and let $k=\ell
  -2$.  The following statements are equivalent.
  \begin{enumerate}
  \item The pattern $p$ is a Baxter permutation.
  \item For each $n\ge k$ and each $k$-element set $H\subseteq [n]$,
    $s_n^H(p)=1$.
  \item For $n=k+3$ and each $k$-element set $H\subseteq [n]$,
    $s_n^H(p)=1$.
  \item There exists $n\ge k+3$ such that for each $k$-element set
    $H\subseteq [n]$, $s_n^H(p)=1$.
  \end{enumerate}
\end{proposition}
\begin{proof}
  Let us first prove that (1) implies (2). Assume that $p$ is a Baxter
  permutation. Choose $n$ and $H$ as in (2). By Lemma~\ref{lem-graph},
  to show that $s_n^H(p)=1$, it is enough to prove that the order
  graph $G_n^H(p)$ has no directed triangles. For contradiction,
  assume that the order graph contains a triangle induced by three
  vertices $h<i<j$.

  Assume that $G_n^H(p)$ contains the edges from $h$ to $i$, from $i$
  to $j$ and from $j$ to $h$ (if the triangle is oriented in the other
  direction, the argument is analogous). Fix $a$, $b$ and $c$, such
  that $h\in I_a$, $i\in I_b$ and $j\in I_c$. Necessarily, $1\le a\le
  b\le c\le k+1=\ell-1$. Note that the three edges $(hi)$, $(ij)$, and
  $(jh)$ imply, respectively, the three inequalities $p_a>p_{b+1}$,
  $p_b>p_{c+1}$, and $p_a<p_{c+1}$. In other words,
  $p_{b+1}<p_a<p_{c+1}<p_b$. This shows that the four indices
  $a,b,b+1,c+1$ are all distinct, and they induce in $p$ a pattern
  order-isomorphic to 2413, contradicting the assumption that $p$ is a
  Baxter permutation. We conclude that for a Baxter permutation $p$,
  the graph $G_n^H(p)$ has no directed triangle, and hence
  $s_n^H(p)=1$.

  Clearly, (2) implies (3) and (3) implies (4). To complete the proof
  of the proposition, we will show that (4) implies (1). Assume that
  $p$ is not a Baxter permutation, and that it contains a copy of 2413
  induced by the indices $a<b<b+1<c+1$ (the case when $p$ contains
  3142 is analogous). In other words, $p$ satisfies
  $p_{b+1}<p_a<p_{c+1}<p_b$. Let $n\ge k+3$ be given. Select a
  $k$-element set $H\subseteq [n]$ in such a way that the three sets
  $I_a$, $I_b$ and $I_c$ are all nonempty. Choose $h\in I_a$, $i\in
  I_b$ and $j\in I_c$ arbitrarily. Necessarily, we have $h<i<j$, and
  the graph $G_n^H(p)$ contains the three directed edges $hi$, $ij$
  and $jh$. This means that $G_n^H(p)$ has a triangle, and hence
  $s_n^H(p)=0$.
\end{proof}

The following result is a direct consequence of
Proposition~\ref{pro-baxter}.

\begin{theorem}\label{thm-baxter}
  Let $p\in \sym_\ell$ be a permutation pattern. Let $k=\ell-2$. If
  $p$ is a Baxter permutation then $s_n^k(p)=\binom{n}{k}$ for each
  $n\ge k$. If $p$ is not a Baxter permutation, then
  $s_n^k(p)<\binom{n}{k}$ for each $n\ge k+3$.  Moreover, all the
  Baxter permutations are strongly $k$-Wilf equivalent.
\end{theorem}

We remark that by a slightly more careful analysis of the proof of
Proposition~\ref{pro-baxter}, we could give a stronger upper bound for
$s_n^k(p)$ when $p$ is not a Baxter permutation. In particular, it is
not hard to see that in such case, $s_n^k(p)$ is eventually a
polynomial in $n$ of degree at most $k-1$, with coefficients depending
on~$k$.

\section{Short patterns}\label{sec:length4}


In the rest of this paper, we focus on explicit formulas for $s_n^k(p)$, 
where $p$ is a pattern of length~$\ell$. We may assume that $k<\ell-1$, and 
$\ell>2$, since for any other values of $(k,\ell)$ the enumeration is trivial 
(see Observation~\ref{obs-l}). We also restrict ourselves to $k\ge 1$, since 
the case $k=0$, which corresponds to classical pattern-avoidance in 
permutations, has already been extensively studied~\cite{bona}.

For a pattern $p$ of length three, the situation is very simple.
Theorem~\ref{thm-baxter} implies that $s_n^1(p)=n$, since all
permutations of length three are Baxter permutations.

Let us now deal with patterns of length four. In Figure~\ref{fig:w4},
we depict the $k$-Wilf equivalence classes, where the four rows, top
to bottom, correspond to the four values $k=0,1,2,3$. Since all the
$k$-Wilf equivalences are closed under complements and reversals (but
not inversions), we represent the 24 patterns of length four by eight
representatives, one from each symmetry class. For instance,
$\{1342,1423\}$ in the second row represents the union of
$\{1342,2431,3124,4213\}$ and $\{1423,2314,3241,4132\}$.

\begin{figure}[htb]
\begin{center}
  \begin{tikzpicture}[join=bevel,
      xscale=0.5, yscale=0.4, font=\small, inner sep=1pt]
    \input{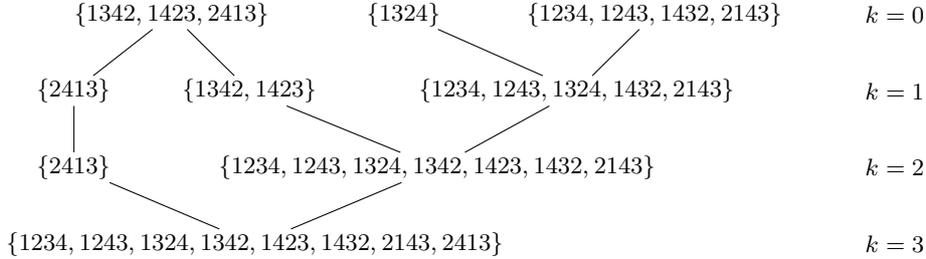}
  \end{tikzpicture}
\end{center}
\caption{The $k$-Wilf-equivalence classes of permutations of size 4.\label{fig:w4}}
\end{figure}

All patterns $p$ of length four except 2413 and 3142 are Baxter permutations, 
and hence they satisfy $s_n^2(p)=\binom{n}{2}$ by Theorem~\ref{thm-baxter}. 

Let us now compute $s_n^2(p)$ for a pattern $p\in\{2413, 3142\}$ and an 
integer~$n$. Since 3142 is the complement of 2413, we know that 
$s_n^2(2413)=s_n^2(2413)$. Let $i$ and $j$ be two indices, with $1\le i<j\le 
n$, and let $H=\{i,j\}$. Let us determine the value of $s_n^H(2413)$. Define 
$I_1$, $I_2$ and $I_3$ as in the previous section, i.e., 
$I_1=\{1,2,\dotsc,i-1\}$, $I_2=\{i+1,\dotsc,j-1\}$ and 
$I_3=\{j+1,\dotsc,n\}$. 

Using the same argument as in the proof of Proposition~\ref{pro-baxter}, we 
deduce that if all the three sets $I_1$, $I_2$, and $I_3$ are nonempty, then 
$s_n^H(2413)=0$. On the other hand, if at least one of the three sets is 
empty, then it is easy to see that the graph $G_n^H(2413)$ is acyclic, and 
hence $s_n^H(2413)=1$ by Lemma~\ref{lem-graph}.

For $n\ge 3$, there are $3n-6$ possibilities to choose $H$ in such a way that 
at least one of the sets $I_1,I_2$ and $I_3$ is empty. We conclude that 
$s_n^2(2413)=s_n^2(3142)=3n-6$.

In the rest of this section, we deal with 1-Wilf equivalence of patterns
of length four, and with the enumeration of the corresponding
avoidance classes.  Theorem~\ref{thm-diag} and symmetry arguments
imply that all the
patterns $1234$, $1243$, $1432$ and $2143$ are strongly
$\star$-Wilf-equivalent, and Theorem~\ref{thm-312} with appropriate
symmetry arguments shows that $1342$
and $1423$ are strongly $\star$-Wilf-equivalent as well. The only case
not covered by these general theorems is the 1-Wilf equivalence of
$1324$ and $1234$, which is handled separately by the next
proposition.

\begin{proposition}\label{pro-1324}
  The patterns 1234 and 1324 are strongly 1-Wilf-equivalent.
\end{proposition}

Let $\pi = \pi_1 \pi_2\dotsb \pi_{j-1}\di\pi_{j+1}\dotsb \pi_n$ be a
partial permutation of length $n$ with a single hole,
appearing at position~$j$. The sequence $\pi_1 \pi_2\dotsb \pi_{j-1}$
will be referred to as the \emph{left part} of $\pi$ and
$\pi_{j+1}\dotsb \pi_n$ will be the \emph{right part} of $\pi$. The
smallest element appearing in the left part of $\pi$ will be called
the \emph{left minimum} of~$\pi$. Left maximum, right minimum and
right maximum are defined analogously.

The following two observations, which follow directly from the
definitions, characterize the avoidance of 1234 and 1324 in partial
permutations with a single hole.

\begin{observation}\label{obs-1234}
  A partial permutation $\pi= \pi_1 \pi_2\dotsb
  \pi_{j-1}\di\pi_{j+1}\dotsb \pi_n$ avoids the pattern 1234 if and
  only if it satisfies the following conditions:
  \begin{enumerate}
  \item The left part of $\pi$ avoids 123.
  \item The elements of the left part that are smaller than the right
    maximum form a decreasing sequence.
  \item The right part of $\pi$ avoids 123.
  \item The elements of the right part that are larger than the left
    minimum form a decreasing sequence.
\end{enumerate}
\end{observation}

\begin{observation}\label{obs-1324}
  A partial permutation $\pi= \pi_1 \pi_2\dotsb
  \pi_{j-1}\di\pi_{j+1}\dotsb \pi_n$ avoids the pattern 1324 if and
  only if it satisfies the following conditions:
  \begin{enumerate}
  \item The left part of $\pi$ avoids 132.
  \item The elements of the left part that are smaller than the right
    maximum form a decreasing sequence.
  \item The right part of $\pi$ avoids 213.
  \item The elements of the right part that are larger than the left
    minimum form a decreasing sequence.
\end{enumerate}
\end{observation}

\setlength{\unitlength}{3mm}
\newcommand{\dotdiam}{0.5}
\begin{figure}
  \begin{center}
    \newcommand{\onepart}{%
      \put(1,6){\line(1,0){5}}
      \put(1,6){\line(0,1){6}}
      \put(1,12){\line(1,0){5}}
      \put(6,6){\line(0,1){6}}
      \multiput(1,10)(1,0){4}{\line(1,1){2}}
      \put(1,11){\line(1,1){1}}
      \put(5,10){\line(1,1){1}}
      \put(2,9){\circle*{\dotdiam}}
      \put(3,8){\circle*{\dotdiam}}
      \put(5,6){\circle*{\dotdiam}}
      \multiput(3.8,7.2)(0.2,-0.2){3}{\line(1,-1){0.1}}
      \put(6.75,7.5){\di}
      \multiput(6,6)(0.6,0){14}{\line(1,0){0.3}}
      \multiput(0,10)(0.6,0){14}{\line(-1,0){0.3}}
      \put(8,4){\line(1,0){5}}
      \put(8,4){\line(0,1){6}}
      \put(8,10){\line(1,0){5}}
      \put(13,4){\line(0,1){6}}
      \multiput(8,4)(1,0){4}{\line(1,1){2}}
      \put(8,5){\line(1,1){1}}
      \put(12,4){\line(1,1){1}}
      \put(9,10){\circle*{\dotdiam}}
      \put(10,8.5){\circle*{\dotdiam}}
      \put(12,6.5){\circle*{\dotdiam}}
      \multiput(10.8,7.7)(0.2,-0.2){3}{\line(1,-1){0.1}}
      \put(1.5,14){\line(1,0){5}}
      \qbezier(1, 13.5)(1, 14)(1.5, 14)
      \qbezier(6.5,14)(7,14)(7,14.5)
      \qbezier(7.5,14)(7,14)(7,14.5)
      \put(7.5,14){\line(1,0){5}}
      \qbezier(12.5,14)(13,14)(13,13.5)
      \put(1.5,2){\line(1,0){1.5}}
      \qbezier(1,2.5)(1,2)(1.5,2)
      \qbezier(3,2)(3.5,2)(3.5,1.5)
      \qbezier(3.5,1.5)(3.5,2)(4,2)
      \put(4,2){\line(1,0){1.5}}
      \qbezier(5.5,2)(6,2)(6,2.5)
      \put(8.5,2){\line(1,0){1.5}}
      \qbezier(8,2.5)(8,2)(8.5,2)
      \qbezier(10,2)(10.5,2)(10.5,1.5)
      \qbezier(10.5,1.5)(10.5,2)(11,2)
      \put(11,2){\line(1,0){1.5}}
      \qbezier(12.5,2)(13,2)(13,2.5)}
    \begin{picture}(14, 16)
      \onepart
      \put(4,15){{\small{avoids $1234$}}}
      \put(1,0){{\small{avoids $123$}}}
      \put(8,0){{\small{avoids $123$}}}
    \end{picture}
    \hskip1cm
    \begin{picture}(14,16)
      \onepart
      \put(4,15){{\small{avoids $1324$}}}
      \put(1,0){{\small{avoids $132$}}}
      \put(8,0){{\small{avoids $213$}}}
    \end{picture}
    \caption{The structures of 1234- and 1324-avoiding partial permutations
      with one hole.\label{figure:avoids1234}}
  \end{center}
\end{figure}
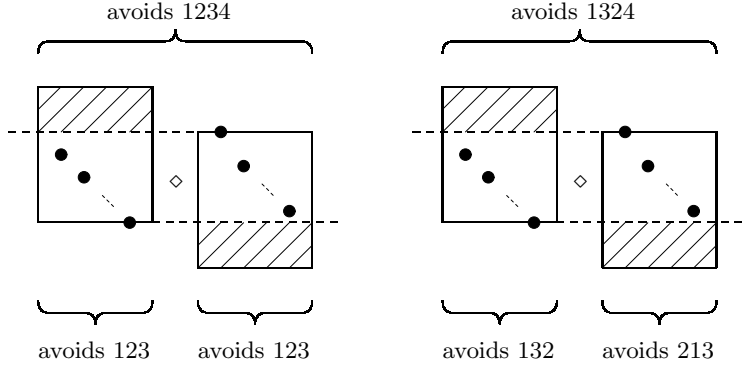

\begin{proof}[Proof of Proposition~\ref{pro-1324}]
  Let us describe a bijection between the sets $\sym_n^1(1234)$ and
  $\sym_n^1(1324)$.  Choose an arbitrary $\pi\in \sym_n^1(1234)$. In
  the first step, we permute the symbols of the left part of $\pi$, so
  that the left part is bijectively transformed from a sequence
  satisfying conditions 1 and 2 of Observation~\ref{obs-1234} to a
  sequence satisfying conditions 1 and 2 of
  Observation~\ref{obs-1324}, while preserving the number of elements
  in the left part that are smaller than the right maximum. Actually,
  we can require a stronger statement, when under the transformation,
  in the left part of $\pi$ the sequence of left-to-right minima will
  be preserved in value and in position, which can be done using the
  Simion-Schmidt bijection; see \cite{SS85} and~\cite[Theorem 1]{CK}.

  In the second step of the bijection, we perform an analogous
  transformation of the right part of the sequence, again using the
  Simion-Schmidt bijection. Indeed, we can achieve a bijective
  transformation between 123- and 213-avoiding permutations
  (corresponding to the right parts) preserving the sequence of
  right-to-left maxima in value and in place: applying reverse and
  complement, it is equivalent to preserving the sequence of
  left-to-right minima in value and in place in a bijection between
  123- and 132-avoiding permutations, which we have by the
  Simion-Schmidt bijection.
\end{proof}

\subsection{Enumeration}\label{sec:enum}

We now focus on explicit enumerations of $s_n^k(p)$ for $p\in\sym_4$ and 
$k=1,2$. In what follows, for two sequences of numbers $\pi_1$ and $\pi_2$
we write $\pi_1<\pi_2$ if each letter of $\pi_1$ is smaller than any
letter of $\pi_2$. Let $C_n$ denote the $n$-th Catalan number $\frac{1}{n+1}\binom{2n}{n}$
and let $C(x)$ be the generating function 
\[
C(x)=\sum_{n\ge 0} C_n x^n=\frac{1-\sqrt{1-4x}}{2x}.
\]

\begin{theorem}\label{enum1} 
For $n\ge 1$, we have $s_n^1(1234)=\binom{2n-2}{n-1}$. 
\end{theorem}
\begin{proof}
  From formula~\eqref{eq:monotone} in Subsection~\ref{ssec:mono},
  we get $s_n^1(1234)=n s_{n-1}^0(123)$, and it is well-known (e.g., see~\cite{CK}) that 
  $s_{n-1}^0(123)= C_{n-1}$. This completes the proof. 
\end{proof}
In Section~\ref{sec:bij}, we also provide a bijective proof of Theorem~\ref{enum1},
by mapping $\sym_n^1(1234)$ to lattice
  paths from (0,0) to $(2n-2,0)$ with steps (1,1) and (1,-1).

\begin{theorem}\label{enum2}
  For $n\geq 1$, we have $s_n^1(1342)=\binom{2n-2}{n-1}-\binom{2n-2}{
    n-5}$.
\end{theorem}

\begin{proof}
  The following observation, coming directly from the definitions,
  characterizes the avoidance of 1342 in partial permutations with a
  single hole.
  \begin{observation}\label{obs-1342}
    A partial permutation $\pi= \pi_1 \pi_2\dotsb \pi_{j-1}\di\pi_{j+1}\dotsb
    \pi_n$ avoids the pattern 1342 if and only if it satisfies the
    following conditions (see also Figure~\ref{figure:avoids1342}):
    \begin{enumerate}
    \item The left part of $\pi$ avoids 123.
    \item The right part of $\pi$ avoids 231.
\item The elements in the right part bigger than the left minimum form an increasing sequence.
\item If $a<b<c$ are three numbers such that $b$ is in the right part of $\pi$ while $a$ and $c$ are in the left part, then $c$ must appear to the left of~$a$ in~$\pi$.
\end{enumerate}
Consequently, the structure of $\pi$ is described by one of the following two cases.
\begin{itemize} 
    \item[(i)] If the right part of $\pi$ is an increasing (possibly empty)
      sequence, then the left part of $\pi$ consists of a decreasing
      sequence of 123-avoiding
      possibly empty blocks as shown on the upper picture in
      Figure~\ref{figure:avoids1342}. 
      Assuming that the right part is of size $k$, $\pi$ can be
      decomposed as $B_1B_2\dotsb B_{k+1}\di a_ka_{k-1}\dotsb a_1$
      with $B_1>a_1>B_2>a_2>\cdots > a_k> B_{k+1}$, where $B_i$ is
      a possibly empty 123-avoiding permutation, for $1\leq i\leq
      k+1$.
    \item[(ii)] 
Suppose the right part of $\pi$ is not an increasing sequence. Let $a$ be the 
smallest symbol in the right part of $\pi$ such that all the symbols in the 
right part greater or equal to $a$ form an increasing sequence (see the lower 
picture in Figure~\ref{figure:avoids1342}). Assuming there are $k$ elements 
greater than $a$ in the right part, $\pi$ can be written as $B_1B_2\dotsb
B_kDC\di AaBa_ka_{k-1}\dotsb a_1$, where $D$, $C$, and all the $B_i$ are 
possibly empty 123-avoiding permutations (we distinguish $D$ and $C$ from the 
$B_i$'s for enumeration arguments below), $A$ and $B$ are 231-avoiding 
permutations with $B$ non-empty, such that
      $$B_1>a_1>B_2>a_2>\cdots >B_k>a_k>D>a>C>B>A.$$
    \end{itemize}
  \end{observation}

  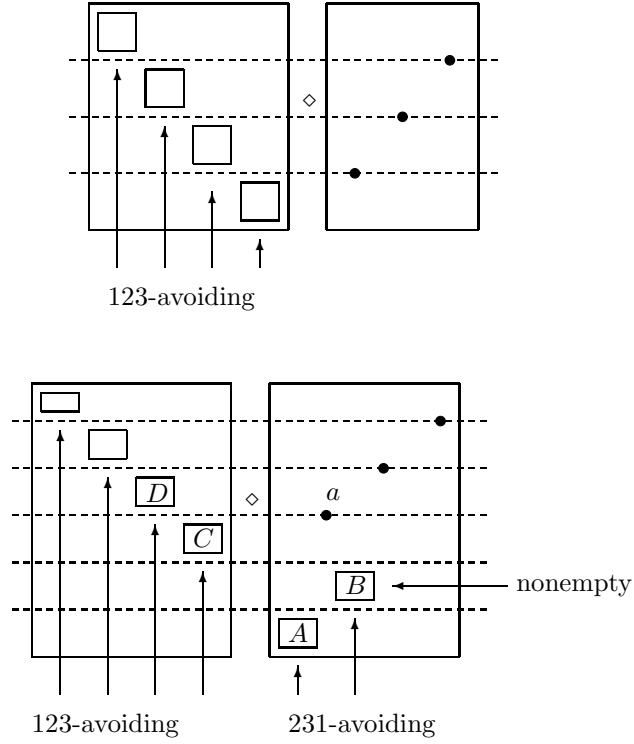
\begin{figure}
    \begin{center}
      \setlength{\unitlength}{2.5mm}
      \begin{picture}(23,16)
        \put(1,4){\line(1,0){10.5}}
        \put(1,4){\line(0,1){12}}
        \put(1,16){\line(1,0){10.5}}
        \put(11.5,4){\line(0,1){12}}
        \put(15,7){\circle*{\dotdiam}}
        \put(17.5,10){\circle*{\dotdiam}}
        \put(20,13){\circle*{\dotdiam}}
        \put(12.25,10.5){\di}
        \multiput(0,7)(0.6,0){38}{\line(1,0){0.3}}
        \multiput(0,10)(0.6,0){38}{\line(1,0){0.3}}
        \multiput(0,13)(0.6,0){38}{\line(1,0){0.3}}
        \put(13.5,4){\line(1,0){8}}
        \put(13.5,4){\line(0,1){12}}
        \put(13.5,16){\line(1,0){8}}
        \put(21.5,4){\line(0,1){12}}
        \newsavebox{\block}
        \savebox{\block}(2,2)[l]{%
          \put(0,0){\line(1,0){2}}
          \put(0,0){\line(0,1){2}}
          \put(0,2){\line(1,0){2}}
          \put(2,0){\line(0,1){2}}}
        \put(1.5,13.5){\usebox{\block}}
        \put(4,10.5){\usebox{\block}}
        \put(6.5,7.5){\usebox{\block}}
        \put(9,4.5){\usebox{\block}}

        \put(2,0){$123$-avoiding}
        \put(2.5,2){\vector(0,1){10.5}}
        \put(5,2){\vector(0,1){7.5}}
        \put(7.5,2){\vector(0,1){4}}
        \put(10,2){\vector(0,1){1.5}}
      \end{picture}
      \vskip10mm
      \begin{picture}(29,18.5)
        \put(1,4){\line(1,0){10.5}}
        \put(1,4){\line(0,1){14.5}}
        \put(1,18.5){\line(1,0){10.5}}
        \put(11.5,4){\line(0,1){14.5}}
        \put(12.25,12){\di}
        \multiput(0,6.5)(0.6,0){42}{\line(1,0){0.3}}
        \multiput(0,9)(0.6,0){42}{\line(1,0){0.3}}
        \multiput(0,11.5)(0.6,0){42}{\line(1,0){0.3}}
        \multiput(0,14)(0.6,0){42}{\line(1,0){0.3}}
        \multiput(0,16.5)(0.6,0){42}{\line(1,0){0.3}}
        \put(13.5,4){\line(1,0){10}}
        \put(13.5,4){\line(0,1){14.5}}
        \put(13.5,18.5){\line(1,0){10}}
        \put(23.5,4){\line(0,1){14.5}}
        \put(16.5,11.5){\circle*{\dotdiam}}
        \put(19.5,14){\circle*{\dotdiam}}
        \put(22.5,16.5){\circle*{\dotdiam}}

        \savebox{\block}(2,1.5)[l]{%
          \put(0,0){\line(1,0){2}}
          \put(0,0){\line(0,1){1.5}}
          \put(0,1.5){\line(1,0){2}}
          \put(2,0){\line(0,1){1.5}}}
        \put(14,4.5){\usebox{\block}}
        \put(17,7){\usebox{\block}}
        \put(9,9.5){\usebox{\block}}
        \put(4,14.5){\usebox{\block}}
        \put(6.5,12){\usebox{\block}}
        \put(1.5,17){\line(1,0){2}}
        \put(1.5,17){\line(0,1){1}}
        \put(1.5,18){\line(1,0){2}}
        \put(3.5,17){\line(0,1){1}}

        \put(14.5,4.8){$A$}\put(17.5,7.3){$B$}\put(9.5,9.8){$C$}
        \put(7,12.2){$D$}\put(16.5,12.3){$a$}
        \put(1,0){$123$-avoiding}
        \put(2.5,2){\vector(0,1){14}}
        \put(5,2){\vector(0,1){11.5}}
        \put(7.5,2){\vector(0,1){9}}
        \put(10,2){\vector(0,1){6.5}}
        \put(14.5,0){$231$-avoiding}
        \put(15,2){\vector(0,1){1.5}}
        \put(18,2){\vector(0,1){4}}
        \put(26.5,7.5){nonempty}
        \put(26,7.75){\vector(-1,0){6}}
      \end{picture}
      \caption{Two possible structures of partial permutations with one hole that
        avoid 1342.\label{figure:avoids1342}}
    \end{center}
  \end{figure}

  Using Observation~\ref{obs-1342}, we will derive a closed-form formula for
  the generating function $\sum_{n\ge 1} s_n^1(1342)x^n$.

It is known that $C(x)$ is the generating function for 
123-avoiding permutations, as well as for 231-avoiding permutations. The partial 
permutations considered in case (i) of
  Observation~\ref{obs-1342} then have the generating function $x\sum_{k\ge 0} x^k 
C^{k+1}(x)=xC(x)/(1-xC(x))$. Note that a factor $x$ in the previous expression 
corresponds to the hole in the partial permutation.     

  On the other hand, the generating function corresponding to case (ii) in
  Observation~\ref{obs-1342} is
  $$\frac{x^2C^3(x)(C(x)-1)}{1-xC(x)}$$ where in the numerator, one
  $x$ corresponds to the hole, the other $x$ corresponds to $a$; $C(x)-1$
  corresponds to the nonempty $B$; $C^3(x)$ corresponds to $A$,
  $C$, and $D$.

  We now sum the two functions and use the well-known relation
  $xC^2(x)=C(x)-1$ to simplify the
  obtained expression:
\begin{align*}
  \frac{x^2C^3(x)(C(x)-1)+xC(x)}{1-xC(x)}&=\frac{x^2C^4(x)(C(x)-1)+xC^2(x)}{C(x)-xC^2(x)}\\
  &=x(C(x)-1)^2C^2(x)+C(x)-1\\
  &=(C(x)-1)(C^2(x)-2C(x)+2).
\end{align*}
From this, we get $s_n^1(1342)=\binom{2n-2}{n-1}-\binom{2n-2}{n-5}$,
corresponding to sequence A026029 in~\cite{sloane} with indices shifted by one.
\end{proof}

\begin{theorem}\label{enum3}
  For $n\geq 1$, we have $s_n^1(2413)=\frac{2}{n+1}\binom{2n}{n}-2^{n-1}$.
\end{theorem}

\begin{proof}
  The following observation, coming directly from the definitions,
  characterizes the avoidance of 2413 in partial permutations with a
  single hole.
  \begin{observation}\label{obs-2413}
    A partial permutation $\pi= \pi_1 \pi_2\dotsb \pi_{j-1}\di\pi_{j+1}\dotsb
    \pi_n$ avoids the pattern 2413 if and only if it satisfies the
    following conditions (see also Figure~\ref{figure:avoids2413}):
    \begin{enumerate}
    \item The left part of $\pi$ avoids 231.
    \item The right part of $\pi$ avoids 312.
\item If $a<b<c$ are three numbers such that $a$ and $c$ appear in the left part of $\pi$ while $b$ 
is in the right part, then $c$ appears to the left of~$a$.
\item If $a<b<c$ are three numbers such that $a$ and $c$ appear in the right part of $\pi$ while $b$ appears in the left part, then $c$ appears to the left of~$a$.
\end{enumerate}
It follows that if both parts of $\pi$ are nonempty, then $\pi$ has one of the following two forms.
\begin{itemize}
\item[(i)] The left minimum is larger than the right maximum, i.e., $\pi=A\di B$ with $A>B$ (see the right picture in Figure~\ref{figure:avoids2413}).

\item[(ii)] Otherwise, the left and the right parts of $\pi$ must consist of decreasing
      sequences of blocks as shown in Figure~\ref{figure:avoids2413} (the left
      picture) where $A$ and $B$ are nonempty; $A$ (resp. $B$) is an
      arbitrary 231- (resp. 312-)avoiding permutation, and the remaining blocks are nonempty decreasing sequences. Moreover, in
      the places indicated by stars, we have possibly empty decreasing
      permutations. Formally speaking, in this case, $\pi$ can be
      decomposed as $\pi=C_0C_1\ldots C_kA\di BD_1D_2\ldots D_{k+1}$
      for some $k$ so that $$C_0>B>C_1>D_1>C_2>D_2>\cdots
      >C_k>D_k>A>D_{k+1},$$ $C_i$'s and $D_i$'s are decreasing
      sequences, $C_0$ and $D_{k+1}$ may be empty, and $A$ and
      $B$ are as described above.
    \end{itemize}
  \end{observation}

  \begin{figure}
    \begin{center}
      \setlength{\unitlength}{3mm}
      \begin{picture}(18,24)
        \put(2,6){\line(1,0){6}}
        \put(2,6){\line(0,1){12}}
        \put(2,18){\line(1,0){6}}
        \put(8,6){\line(0,1){12}}
        \put(10,6){\line(1,0){6}}
        \put(10,6){\line(0,1){12}}
        \put(10,18){\line(1,0){6}}
        \put(16,6){\line(0,1){12}}

        \newsavebox{\hcara}
        \savebox{\hcara}(18,0.1)[l]{%
          \multiput(0,0)(0.6,0){30}{\line(1,0){0.3}}}
        \put(0,8){\usebox{\hcara}}
        \put(0,10){\usebox{\hcara}}
        \put(0,12){\usebox{\hcara}}
        \put(0,14){\usebox{\hcara}}
        \put(0,16){\usebox{\hcara}}
        \newsavebox{\vcara}
        \savebox{\vcara}(0.1,16)[l]{%
          \multiput(0,0)(0,0.6){28}{\line(0,1){0.3}}}
        \put(4,4){\usebox{\vcara}}
        \put(6,4){\usebox{\vcara}}
        \put(12,4){\usebox{\vcara}}
        \put(14,4){\usebox{\vcara}}

        \qbezier(0,3)(0,2.5)(0.5,2.5)
        \put(0.5,2.5){\line(1,0){3}}
        \qbezier(3.5,2.5)(4,2.5)(4,2)
        \qbezier(4,2)(4,2.5)(4.5,2.5)
        \put(4.5,2.5){\line(1,0){3}}
        \qbezier(7.5,2.5)(8,2.5)(8,3)
        \put(1,0){avoids $231$}
        \qbezier(10,3)(10,2.5)(10.5,2.5)
        \put(10.5,2.5){\line(1,0){3}}
        \qbezier(13.5,2.5)(14,2.5)(14,2)
        \qbezier(14,2)(14,2.5)(14.5,2.5)
        \put(14.5,2.5){\line(1,0){3}}
        \qbezier(17.5,2.5)(18,2.5)(18,3)
        \put(11,0){avoids $312$}
        \qbezier(0,21)(0,21.5)(0.5,21.5)
        \put(0.5,21.5){\line(1,0){8}}
        \qbezier(8.5,21.5)(9,21.5)(9,22)
        \qbezier(9.5,21.5)(9,21.5)(9,22)
        \put(9.5,21.5){\line(1,0){8}}
        \qbezier(17.5,21.5)(18,21.5)(18,21)
        \put(6,23){avoids $2413$}

        \put(8.5,10.5){{\LARGE{\di}}}
        \put(16.5,4){{\LARGE{*}}}
        \put(0.5,18){{\LARGE{*}}}
        \savebox{\block}(1.5,1.5)[l]{%
          \put(0,0){\line(1,0){1.5}}%
          \put(0,0){\line(0,1){1.5}}%
          \put(0,1.5){\line(1,0){1.5}}%
          \put(1.5,0){\line(0,1){1.5}}}
        \put(6.25,6.25){\usebox{\block}}
        \put(6.5,6.5){$A$}
        \put(10.25,16.25){\usebox{\block}}
        \put(10.5,16.5){$B$}
        \savebox{\block}(1.5,1.5)[l]{%
          \put(0,0){\line(1,0){1.5}}%
          \put(0,0){\line(0,1){1.5}}%
          \put(0,1.5){\line(1,0){1.5}}%
          \put(0,1.5){\line(1,-1){1.5}}%
          \put(1.5,0){\line(0,1){1.5}}}
        \put(4.25,10.25){\usebox{\block}}
        \put(2.25,14.25){\usebox{\block}}
        \put(12.25,12.25){\usebox{\block}}
        \put(14.25,8.25){\usebox{\block}}
      \end{picture}
      \hskip15mm
      \begin{picture}(12,24)
        \put(0,7){\line(1,0){5}}
        \put(0,7){\line(0,1){10}}
        \put(0,17){\line(1,0){5}}
        \put(5,7){\line(0,1){10}}
        \put(7,7){\line(1,0){5}}
        \put(7,7){\line(0,1){10}}
        \put(7,17){\line(1,0){5}}
        \put(12,7){\line(0,1){10}}
        \put(5.5,11.5){{\LARGE{\di}}}
        \savebox{\block}(4,4)[l]{%
          \put(0,0){\line(1,0){4}}%
          \put(0,0){\line(0,1){4}}%
          \put(0,4){\line(1,0){4}}%
          \put(4,0){\line(0,1){4}}}
        \put(0.5,12.5){\usebox{\block}}
        \put(7.5,7.5){\usebox{\block}}
        \put(2,14){$A$}
        \put(9,9){$B$}
        \qbezier(0,6)(0,5.5)(0.5,5.5)
        \put(0.5,5.5){\line(1,0){1.5}}
        \qbezier(2,5.5)(2.5,5.5)(2.5,5)
        \qbezier(2.5,5)(2.5,5.5)(3,5.5)
        \put(3,5.5){\line(1,0){1.5}}
        \qbezier(4.5,5.5)(5,5.5)(5,6)
        \put(0,3){avoids $231$}
        \qbezier(7,6)(7,5.5)(7.5,5.5)
        \put(7.5,5.5){\line(1,0){1.5}}
        \qbezier(9,5.5)(9.5,5.5)(9.5,5)
        \qbezier(9.5,5)(9.5,5.5)(10,5.5)
        \put(10,5.5){\line(1,0){1.5}}
        \qbezier(11.5,5.5)(12,5.5)(12,6)
        \put(7,3){avoids $312$}
        \qbezier(0,18)(0,18.5)(0.5,18.5)
        \put(0.5,18.5){\line(1,0){5}}
        \qbezier(5.5,18.5)(6,18.5)(6,19)
        \qbezier(6,19)(6,18.5)(6.5,18.5)
        \put(6.5,18.5){\line(1,0){5}}
        \qbezier(11.5,18.5)(12,18.5)(12,18)
        \put(3,20){avoids $2413$}
      \end{picture}
      \caption{Two possible structures of partial permutations with one hole that
        avoid 2413.\label{figure:avoids2413}}
    \end{center}
  \end{figure}
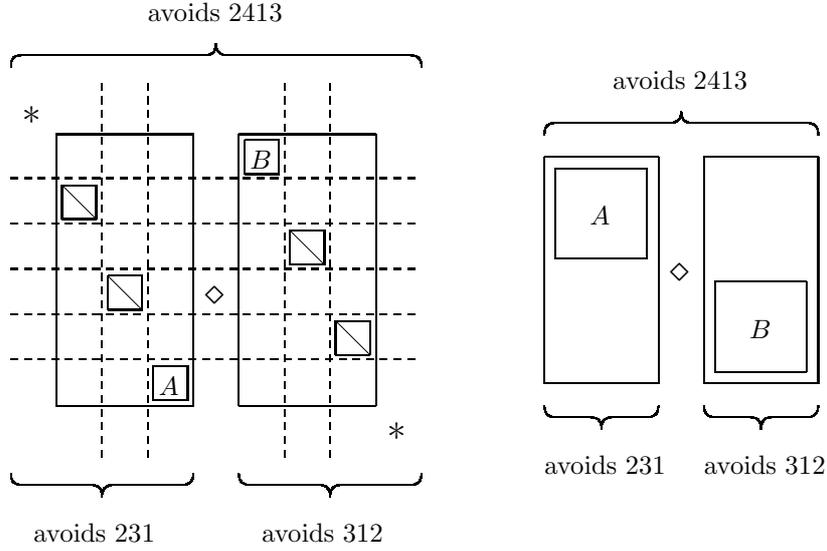

  Let us derive the generating function based on Observation~\ref{obs-2413}.

  If both the left and the right parts of $\pi$ are empty, the
      corresponding generating function is $x$; if exactly one of the parts is empty,
      the generating function is $x(C(x)-1)$. Together, these cases have
      generating function $x(2C(x)-1)$. In what follows, we assume the parts
      are not empty.

The generating function for case (i) in Observation~\ref{obs-2413} is
clearly $x(C(x)-1)^2$.

Consider case (ii) in Observation~\ref{obs-2413}. The generating function for a nonempty
decreasing block is
      $\frac{x}{1-x}$, whereas the generating function for a possibly empty such
      block is $\frac{1}{1-x}$.  Thus, since the number of decreasing
      blocks in the left part is the same as that in the right part
      (not counting the places indicated by the stars), the number of
      partial permutations in this case  has the following generating function (an
      extra $x$ corresponds to the hole):
      $$\frac{x}{(1-x)^2}\frac{(C(x)-1)^2}{1-\left(\frac{x}{1-x}\right)^2}
      =\frac{x(C(x)-1)^2}{1-2x}.
      $$

  Summing the cases above, we see that the generating function for $s_n^1(2413)$ is
  $$x(2C(x)-1)+x( C(x)-1)^2+ \frac{x( C(x)-1)^2}{1-2x}=2C(x)-\frac{x}{1-2x}-2,
  $$ which gives
  $s_n^1(2413)=\frac{2}{n+1} \binom{2n}{n}-2^{n-1}$.
\end{proof}

\subsection{Bijective proof of Theorem~\ref{enum1}}\label{sec:bij} 

Theorem~\ref{enum1} states that $s_n^1(1234)=\binom{2n-2}{n-1}$. We
provide a bijective proof of this fact here.

\begin{theorem} 
  There is a bijection between partial permutations of length $n$ with one
  hole that avoid 1234, and the set of all lattice paths from $(0,0)$ to
  $(2n-2,0)$ with steps $(1,1)$ and $(1,-1)$.
\end{theorem}

\begin{proof} 
  Our proof is based on a known bijective proof of the fact that the
  number of Dyck paths of length $2n$ is given by the $n$-th Catalan
  number $\frac{1}{n+1}\binom{2n}{n}$.

  Let $\pi\in \sym_n^1(1234)$ and let the hole be in position $i$, $1\leq i\leq
  n$.  Remove the hole and map the obtained 123-avoiding permutation of
  length $n-1$, using any of your favorite bijections~\cite{CK} to a
  Dyck path $P$ from $(0,0)$ to $(2n-2,0)$. Now add a down-step at the
  end of $P$. Thus, $P$ has $n$ down-steps and $n-1$ up-steps. Cut $P$
  into two parts: $P_1$ is the (nonempty) part of all steps to the
  left, but {\em not} including the $i$-th down-step, and $P_2$ is the
  remaining, (nonempty) part. Move $P_2$ so that it starts from $(0,0)$
  and append $P_1$ to it. We now have a path from $(0,0)$ to $(2n-1,-1)$
  inducing, in an injective way, the desired path of length $2n-2$
  from $(1,-1)$ to $(2n-1,-1)$.

  The reverse to the procedure above is easy to see: append an extra
  down-step to the left of a given path from $(0,0)$ to $(2n-2,0)$ and
  shift the obtained path to start at $(0,0)$; find the leftmost minimum
  of the new path, cutting it into two parts and reassembling. Thus we
  get a bijection. In Figure \ref{figure:bijection} we provide an
  example using Krattenthaler's bijection from 123-avoiding
  permutations to Dyck paths (see~\cite{CK}).
\end{proof}

\setlength{\unitlength}{2mm}
\newcommand{\vertexdiam}{0.4}
\newcounter{cx}
\setcounter{cx}{0}
\newcounter{cy}
\setcounter{cy}{0}
\newcommand{\lineup}{%
  \put(\thecx,\thecy){\circle*{\vertexdiam}}%
  \put(\thecx,\thecy){\line(1,1){1}}%
  \addtocounter{cx}{1}\addtocounter{cy}{1}}
\newcommand{\linedown}{%
  \put(\thecx,\thecy){\circle*{\vertexdiam}}%
  \put(\thecx,\thecy){\line(1,-1){1}}%
  \addtocounter{cx}{1}\addtocounter{cy}{-1}}
\newcommand{\setcxyto}[2]{\setcounter{cx}{#1}\setcounter{cy}{#2}}
\newcommand{\finalpt}{\put(\thecx,\thecy){\circle*{\vertexdiam}}}
\begin{figure}
\begin{center}
\begin{picture}(52, 29)
  \put(4, 22){5 4 2 \di{} 8 7 6 1 3}
  \put(24,22){\vector(1,0){4}}
  \put(24,5){\vector(1,0){4}}
  \put(28,16){\vector(-1,-1){4}}
  \put(31, 22){\vector(1, 0){21}}
  \put(32, 17){\vector(0, 1){12}}
  \setcxyto{32}{22}
  \lineup \lineup \lineup \lineup \linedown \lineup
  \linedown \lineup \linedown \linedown \linedown
  \lineup \lineup \linedown \linedown \linedown \linedown
  \finalpt
  \multiput(41,17.5)(0,1){11}{\line(0,1){0.5}}

  \put(0, 5){\vector(1, 0){21}}
  \put(1, 0){\vector(0, 1){12}}
  \setcxyto{1}{5}
  \linedown \linedown \lineup \lineup \linedown \linedown
  \linedown \linedown \lineup \lineup \lineup \lineup
  \linedown \lineup \linedown \lineup \linedown
  \finalpt

  \put(31, 5){\vector(1, 0){21}}
  \put(32, 0){\vector(0, 1){12}}
  \setcxyto{32}{5}
  \linedown \lineup \lineup \linedown \linedown \linedown
  \linedown \lineup \lineup \lineup \lineup \linedown
  \lineup \linedown \lineup \linedown
  \finalpt
\end{picture}
\caption{An example of a bijective map of $\pi\in \sym_9^1(1234)$ to a
lattice path from (0,0) to (16,0).\label{figure:bijection}}
\end{center}
\end{figure}
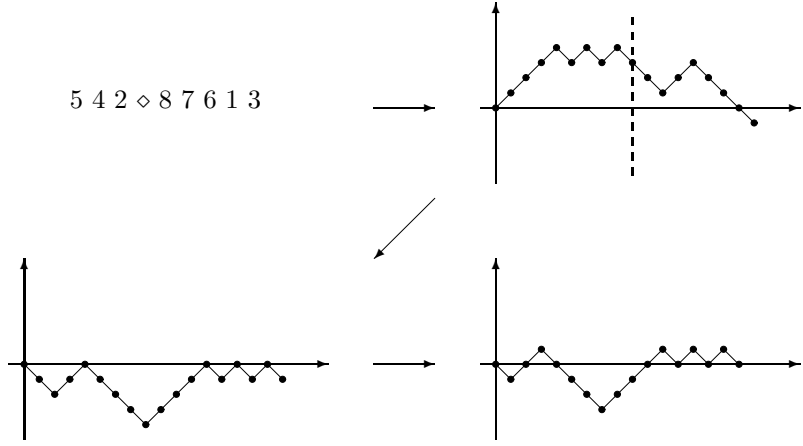

\section{Directions of further research}\label{sec:further}

We have shown that classical Wilf equivalence may be regarded as a special 
case in a hierarchy of $k$-Wilf equivalence relations, and that many 
properties previously established in the context of Wilf equivalence can be 
generalized to all the $k$-Wilf equivalences. In many situations, 
understanding the $k$-Wilf equivalence class of a given pattern $p$ becomes 
easier as $k$ increases. Consider, for example, the identity permutation 
$\id_\ell=123\dotsb\ell$. We know that the $(\ell-1)$-Wilf equivalence class of $\id_\ell$ 
contains every permutation of length $\ell$, and we have shown that the 
$(\ell-2)$-Wilf equivalence class of $\id_\ell$ contains 
exactly the Baxter permutations of length~$\ell$. What is the $(\ell-3)$-Wilf 
equivalence class of~$\id_\ell$? For $\ell=3$ and $\ell=4$ 
that class contains exactly the layered permutations of length~$\ell$. 
Computer enumeration suggests that the same is true for $\ell=5$. We do not know 
whether this behaviour generalizes to larger values of~$\ell$.


Another natural direction of further research is to extend known (general) 
results in permutation patterns theory to the setting of partial 
permutations. For example, it is natural to investigate the growth rate of 
$s_n^k(p)$ for a fixed $k$ and a fixed pattern $p$, with $n\to \infty$. In 
the setting of non-partial permutations, Marcus and Tardos~\cite{MaTa} have 
shown that for each pattern $p$, there is a constant $K_p$ (known as the 
Stanley--Wilf limit of~$p$), such that 
\[
\lim_{n\to\infty} \sqrt[n]{s^0_n(p)}= K_p \text{ or equivalently } s^0_n(p)=K_p^{n+o(n)}.
\]
For a pattern $p$ of length $\ell$, a result of Valtr cited by Kaiser and 
Klazar~\cite{kakl} shows that $K_p\ge \Omega(\ell^2)$, while the best known 
general upper bound, due to Cibulka~\cite{Cibulka}, is of order~$2^{O(\ell\log 
\ell)}$. It is also easy to get a lower bound $K_p\ge \ell-1$
(see~\cite[Page 167, exc. 33]{bona}).

We do not know whether the limit $\lim_{n\to\infty} \sqrt[n]{s_n^k(p)}$ 
exists for every $p$ and~$k$. We can, however, bound the growth of $s_n^k(p)$ 
in terms of the Stanley--Wilf limits of certain subpermutations of~$p$. To 
make this specific, let us introduce the following terminology: for two 
permutations $p\in\sym_n$ and $q\in\sym_m$, we say that $q$ is a 
\emph{consecutive subpattern of $p$} if for some $i\ge 0$ the consecutive 
subsequence $p_{i+1}, p_{i+2},\dotsc, p_{i+m}$ of $p$ is order-isomorphic 
to~$q$. 

\begin{theorem}\label{thm-number}
Let $k\ge0$ be an integer. Let $p$ be a permutation pattern of length $\ell$, 
with $\ell>k$. Let $q$ be the consecutive subpattern of $p$ of length 
$\ell-k$ chosen in such a way that its Stanley--Wilf limit $K_q$ is as large 
as possible. We then have the bounds
\[
K_q^{n+o(n)}\le s_n^k(p) \le (k+1)^n K_q^{n+o(n)}. 
\]
\end{theorem}

\begin{proof}      
Let us first prove the lower bound. Suppose that $p$ has an 
occurrence of $q$ at positions $i+1,i+2,\dotsc,i+\ell-k$, for some $i\ge 0$. 
Choose an $n\ge k$. Consider a partial permutation $\pi\in \sym_n^k$ that 
begins with $i$ holes, followed by a (non-partial) permutation $\pi'$ 
of length $n-k$, followed by $k-i$ holes. It is easy to see that 
$\pi$ avoids $p$ if and only if $\pi'$ avoids $q$, which means that 
$s_n^k(p)\ge s^0_{n-k}(q)$. This implies the desired lower bound.

To prove the upper bound, we fix an arbitrary $\varepsilon>0$, and we will 
show that $s_n^k(p)\le C \binom{n}{k} (k+1)^{n-k} (K_q+\varepsilon)^n$ for some 
$C$ depending on $p$, $k$ and $\varepsilon$, but not on~$n$. From this, the upper 
bound will easily follow.

Choose again an arbitrary $n\ge k$ and fix a set $H\subseteq[n]$ of size $k$. 
Let us estimate the size of $s_n^H(p)$. Let $i_1<i_2<\dotsb<i_k$ be the 
elements of~$H$. Let us also define $i_0=0$ and $i_{k+1}=n+1$. Each partial 
permutation~$\pi\in s_n^H(p)$ can be written as
\[
\pi=\pi^{(1)}\di\pi^{(2)}\di\dotsb\di\pi^{(k)}\di\pi^{(k+1)},
\]
where $\pi^{(j)}$ is a (possibly empty) subsequence of $\pi$ of length 
$n_j=i_j-i_{j-1}-1$ that does not contain any hole. Since $\pi$ avoids 
$p$, it is clear that $\pi^{(j)}$ must avoid the consecutive subpattern 
$q^{(j)}$ of $p$ that appears in $p$ at positions $j,j+1,\dotsc, j+\ell-k-1$. In 
other words, $\pi^{(j)}$ must be order-isomorphic to a $q^{(j)}$-avoiding 
permutation $\sigma^{(j)}$ of the set~$[n_j]$. 

Note that to describe a partial permutation $\pi\in s_n^H(p)$ uniquely, it 
is enough to specify for every $j\in[k+1]$ the $q^{(j)}$-avoiding permutation 
$\sigma^{(j)}$ of size $n_j$, and then, for each number $x\in[n-k]$, to 
specify which of the $k+1$ subsequences $\pi^{(j)}$ contains the value~$x$. 

Since each $q^{(j)}$ has Stanley--Wilf limit at most $K_q$, there is a 
constant $Q$ (depending on $p$, $k$ and $\varepsilon$) such that for every 
$m\in\NN$ and every $j\in[k+1]$, there are at most $Q(K_q+\varepsilon)^m$ 
permutations of $[m]$ that avoid~$q^{(j)}$. Thus, $s^H_n(p)\le 
Q^{k+1}(k+1)^{n-k}(K_q+\varepsilon)^n$. Since there are $\binom{n}{k}$ 
possibilities for $H$, we get the desired bound for~$s_n^k(p)$.
\end{proof}

We remark that for all the pattern-avoiding classes for which we can provide 
an enumeration, the limit $\lim_{n\to\infty} \sqrt[n]{s_n^k(p)}$ exists and is 
equal to the value $K_q$ from Theorem~\ref{thm-number}. This means that the 
lower bound from Theorem~\ref{thm-number} in general cannot be improved.

We close the section by summarizing the main open problems.

\begin{enumerate}
\item 
Find a combinatorial proof for the formulas for  $s_n^1(1342)$ and 
$s_n^1(2413)$ derived in Theorems~\ref{enum2} and~\ref{enum3}.
\item                                                                 
Which permutations are $k$-Wilf equivalent to~$\id_{k+3}$? Are they the 
layered permutations of length $k+3$? 
\item
We know that $s_n^H(\id_{k+3})=C_{n-k}$ for any $n\ge k\ge 0$ and any set 
$H\subseteq[n]$ of size $k$, where $C_m$ is the $m$-th Catalan number. Can we 
have $s_n^H(p)>C_{n-k}$ for some permutation $p$ of length $k+3$, some 
set $H$ of size $k$ and some~$n\ge k$? Can we even have 
$s_n^k(p)>s_n^k(\id_{k+3})$ for some $p\in \sym_{k+3}$?
\item 
Does the limit $\lim_{n\to\infty}\sqrt[n]{s_n^k(p)}$ exist for each $k$ and~$p$? 
Is the limit equal to the value $K_q$ defined in Theorem~\ref{thm-number}? Can the 
upper bound in Theorem~\ref{thm-number} be improved?
\end{enumerate}


\bibliography{papp}
\bibliographystyle{abbrv}

\end{document}